\documentclass{article}

\usepackage{arxiv}

\usepackage[utf8]{inputenc} %
\usepackage[T1]{fontenc}    %
\usepackage{hyperref}       %
\usepackage{url}            %
\usepackage{booktabs}       %
\usepackage{nicefrac}       %
\usepackage{microtype}      %
\usepackage[utf8]{inputenc}
\usepackage{amsmath,amsfonts,amssymb,amsthm,bbm,graphicx,enumerate,times}
\usepackage{mathtools}
\usepackage[usenames,dvipsnames]{color}
\hypersetup{
	breaklinks,
	colorlinks,
	linkcolor=gray,
	citecolor=gray,
	urlcolor=gray,
}
\usepackage[table,usenames,dvipsnames]{xcolor}
\usepackage{wrapfig}
\usepackage{tikz}
\usepackage{float}
\usepackage{comment}
\usepackage{csquotes}
\usepackage{todonotes}
\usepackage[ruled]{algorithm2e}
\usepackage{caption} 
\theoremstyle{plain}
\newtheorem{theorem}{Theorem}[section]

\newtheorem{lemma}[theorem]{Lemma}
\newtheorem{proposition}[theorem]{Proposition}

\theoremstyle{definition}
\newtheorem{definition}{Definition}[section]
\newtheorem{example}[definition]{Example}
\newtheorem{remark}[definition]{Remark}

\newcommand{\rmapleqq}[2]{\tau^\leq_{#1} (#2)}

\newcommand{\rmapgeqq}[2]{\tau^\geq_{#1} (#2)}

\newcommand{\opleq}[2]{#1^\leq_{#2}}

\newcommand{\opgeq}[2]{#1^\geq_{#2}}

\DeclareMathOperator*{\argmin}{argmin}

\title{A block-sparse Tensor Train Format for sample-efficient high-dimensional Polynomial Regression}

\author{
 Michael G\"otte \\
  TU Berlin\\
  goette@math.tu-berlin.de
   \And
    Reinhold Schneider\\
    TU Berlin\\
  schneidr@math.tu-berlin.de
  \And
   Philipp Trunschke\\
   TU Berlin\\
  ptrunschke@mail.tu-berlin.de
}

\begin{document}
\maketitle
\begin{abstract}
Low-rank tensors are an established framework for high-dimensional least-squares problems.
We propose to extend this framework by including the concept of block-sparsity.
In the context of polynomial regression each sparsity pattern corresponds to some subspace of homogeneous multivariate polynomials.
This allows us to adapt the ansatz space to align better with known sample complexity results.
The resulting method is tested in numerical experiments and demonstrates improved computational resource utilization and sample efficiency.
\end{abstract}


\section{Introduction}

An important problem in many applications is the identification of a function from measurements or random samples.
For this problem to be well-posed, some prior information about the function has to be assumed and a common requirement is that the function can be approximated in a finite dimensional ansatz space.
For the purpose of extracting governing equations the most famous approach in recent years has been SINDy~\cite{brunton_discovering_2016}.
However, the applicability of SINDy to high-dimensional problems is limited since truly high-dimensional problems require a nonlinear parameterization of the ansatz space.
One particular reparametrization that has proven itself in many applications are tensor networks.
These allow for a straight-forward extension of SINDy~\cite{gels_multidimensional_2019} but can also encode additional structure as presented in~\cite{goesmann_tensor_2020-1}.
The compressive capabilities of tensor networks originate from this ability to exploit additional structure like smoothness, locality or self-similarity and have hence been used in solving high-dimensional equations~\cite{kazeev_low-rank_2012-1,kazeev_quantized_2018,bachmayr_stability_2020,Eigel_2016}.
In the context of optimal control tensor train networks have been utilized for solving the Hamilton--Jacobi--Bellman equation in~\cite{dolgov_tensor_2021,oster_approximating_2020}, for solving backward stochastic differential equations in~\cite{richter_solving_2021} and for the calculation of stock options prices in~\cite{bayer_pricing_2021,glau_low-rank_2020}.
In the context of uncertainty quantification they are used in~\cite{eigel_non-intrusive_2019,eigel_variational_2019,Zhang2015TTUQ} and in the context of image classification they are used in~\cite{klus_tensor-based_2019-1,stoudenmire_supervised_2016}.\\
A common thread in these publications is the parametrization of a high-dimensional ansatz space by a tensor train network which is then optimized.
In most cases this means that the least-squares error of the parametrized function to the data is minimized. 
There exist many methods to perform this minimization.
A well-known algorithm in the mathematics community is the \emph{alternating linear scheme (ALS)}~\cite{oseledets_dmrg_2011,holtz_alternating_2012}
, which is related to the famous DMRG method~\cite{white_density_1992} for solving the Schr\"odinger equation in Quantum Physics.
Although, not directly suitable for recovery tasks, it became apparent that DMRG and ALS can be adapted to work in this context.
Two of these extensions to the ALS algorithm are the \emph{stablilized ALS approximation (SALSA)}~\cite{grasedyck_stable_2019} and the \emph{block alternating steepest descent for Recovery (bASD)} algorithm~\cite{eigel_non-intrusive_2019}.
Both adapt the tensor network ranks and are better suited to the problem of data identification.
Since the set of tensor trains of fixed rank forms a manifold~\cite{holtz_manifolds_2012} it is also possible to perform gradient based optimization schemes~\cite{steinlechner_riemannian_2016}.
This however is not a path that we pursue in this work.
Our contribution extends the ALS (and SALSA) algorithm and can be applied to many of the fields stated above.\\
In~\cite{eigel_convergence_2020} an important observation was made.
It was shown that for tensor networks the sample complexity, meaning the number of data points needed, is related to the dimension of the high-dimensional ansatz space.
For the task of data identification, this is a somewhat disappointing result, since it states that the worst-case bound for a low-rank tensor network is similar to the bound for the full high-dimensional tensor space.
But strangely enough, these huge sample sizes are not needed in most practical examples.
We take this results as a foundation to rethink the approach at hand.
By restricting the full tensor space to a subspace for which the worst-case sample complexity is more moderate we can reduce the gap between observed sample complexity and proven worst-case bound.
By doing so we do not have to worry that our application may require a huge number of samples.
In this work we consider ansatz spaces of homogeneous polynomials of a given degree.
These spaces exhibits a more favorable sampling complexity than the full multivariate polynomial tensor space and there exist many approximation theoretic results that ensure a good approximation with a low degree polynomial for many classes of functions.
Both properties are important to recover a function from data.
The presented approach is very versatile and can be combined with many polynomial approximation strategies like the use of Taylor's theorem in~\cite{breiten_taylor_2019}.
\\
A central objective of this paper is to restrict the linear ansatz space while retaining the profitable compression of tensor product representations.
This can be achieved because the coefficient tensor of a homogeneous polynomial can be represented as a tensor train with a block-sparse structure in its component tensors.
This representation allows us to parametrize the space of homogeneous polynomials of a given degree in an exact and very efficient manner.
This is known to quantum physicists for at least a decade~\cite{singh_tensor_2010-1} but it was introduced to the mathematics community only recently in~\cite{bachmayr}.
In the language of quantum mechanics one would say that %
there exists an operator for which the coefficient tensor of any homogeneous polynomial is an eigenvector.
This encodes a symmetry, where the eigenvalue of this eigenvector is the degree of the homogeneous polynomial, which acts as a quantum number
and corresponds to the particle number of bosons and fermions.
\\
This means, we kill two birds with one stone.
By applying block-sparsity to the coefficient tensor we can restrict the ansatz space to well-behaved functions which can be identified with a reasonable sample size.
At the same time we reduce the number of parameters and speed up the least-squares minimization task.
\\
The remainder of this work is structured as follows.
Section~\ref{sec:notation} introduces basic tensor notation, the different parametrizations of polynomials that are used in this work and then formulates the associated least-squares problems.
In Section~\ref{sec:knownresults} we state the known results on sampling complexity and block sparsity.
Furthermore, we set the two results in relation and argue why this leads to more favorable ansatz spaces.
This includes a proof of rank-bounds for a class of homogeneous polynomials which can be represented particularly efficient as tensor trains. 
Section~\ref{sec:method} derives two parametrizations from the results of Section~\ref{sec:knownresults} and presents the algorithms that are used to solve the associated least-squares problems.
Finally, Section~\ref{sec:numerics} gives some numerical results for different classes of problems focusing on the comparison of the sample complexity for the full- and sub-spaces.
Most notably, the recovery of a quantity of interest for a parametric PDE, where our approach achieves successful recovery with relatively few parameters and samples.
We observed that for suitable problems the number of parameters can be reduced by a factor of almost $10$.

\section{Notation}\label{sec:notation}

In our opinion,  using a graphical notation for the involved contractions in a tensor network drastically simplifies the expressions makes the whole setup more approachable.
This Section introduces this graphical notation for tensor networks, the spaces that will be used in the remainder of this work and the regression framework.

\subsection{Tensors and indices}
\begin{definition}
    Let $d\in\mathbb{N}_{>0}$. Then $\boldsymbol{n} = (n_1,\cdots,n_d)\in\mathbb{N}^d$ is called a \emph{dimension tuple of order $d$} and $x\in\mathbb{R}^{n_1\times\cdots\times n_d} =: \mathbb{R}^{\boldsymbol{n}}$ is called a \emph{tensor of order $d$ and dimension $\boldsymbol{n}$}.
    Let $\mathbb{N}_{n} = \{1,\ldots,n\}$ then a tuple $(l_1,\ldots,l_d) \in \mathbb{N}_{n_1}\times\cdots\times\mathbb{N}_{n_d}=:\mathbb{N}_{\boldsymbol{n}}$ is called a \emph{multi-index} and the corresponding entry of $x$ is denoted by $x(l_1,\ldots,l_d)$.
    The positions $1,\ldots,d$ of the indices $l_1,\ldots,l_d$ in the expression $x(l_1,\ldots,l_d)$ are called \emph{modes of $x$}.
\end{definition}

To define further operations on tensors it is often useful to associate each mode with a symbolic index.

\begin{definition}
	A \emph{symbolic index} $i$ of dimension $n$ is a placeholder for an arbitrary but fixed natural number between $1$ and $n$.
	For a dimension tuple $\boldsymbol{n}$ of order $d$ and a tensor $x\in\mathbb{R}^{\boldsymbol{n}}$ we may write $x(i_1,\ldots,i_d)$ and tacitly assume that $i_k$ are indices of dimension $n_k$ for each $k=1,\ldots,d$.
    When standing for itself this notation means $x(i_1,\ldots,i_d) = x\in\mathbb{R}^{\boldsymbol{n}}$ and may be used to \emph{slice} the tensor
    \begin{equation*}
        x(i_1,l_2,\ldots,l_d) \in \mathbb{R}^{n_1}
    \end{equation*}
    where $l_k\in\mathbb{N}_{n_k}$ are fixed indices for all $k=2,\ldots,d$.
    For any dimension tuple $\boldsymbol{n}$ of order $d$ we define the symbolic multi-index $i^{\boldsymbol{n}} = (i_1, \ldots, i_d)$ of dimension $\boldsymbol{n}$ where $i_k$ is a symbolic index of dimension $n_k$ for all $k=1,\ldots,d$.
\end{definition}

\begin{remark}
	We use the letters $i$ and $j$ (with appropriate subscripts) for symbolic indices while reserving the letters $k$, $l$ and $m$ for ordinary indices.
\end{remark}
\begin{example}
	Let $x$ be an order $2$ tensor with mode dimensions $n_1$ and $n_2$, i.e.\ an $n_1$-by-$n_2$ matrix.
	Then $x(\ell_1,j)$ denotes the $\ell_1$-th row of $x$ and $x(i,\ell_2)$ denotes the $\ell_2$-th column of $x$.
\end{example}

Inspired by Einstein notation we use the concept of symbolic indices to define different operations on tensors.
\begin{definition}\label{def:indexproduct}
	Let $i_1$ and $i_2$ be (symbolic) indices of dimension $n_1$ and $n_2$, respectively and let $\varphi$ be a bijection
	\begin{equation*}
	     \varphi : \mathbb{N}_{n_1}\times \mathbb{N}_{n_2} \rightarrow \mathbb{N}_{n_1n_2}.
	\end{equation*}
	We then define \textit{the product of indices} with respect to $\varphi$ as $j=\varphi(i_1,i_2)$ where $j$ is a (symbolic) index of dimension $n_1n_2$.
	In most cases the choice of bijection is not important and we will write $i_1\cdot i_2 := \varphi(i_1,i_2)$ for an arbitrary but fixed bijection $\varphi$.
	For a tensor $x$ of dimension $(n_1,n_2)$ the expression
	\begin{equation*}
	    y(i_1\cdot i_2) = x(i_1, i_2)
	\end{equation*}
	defines the tensor $y$ of dimension $(n_1 n_2)$ while the expression
	\begin{equation*}
	    x(i_1, i_2) =  y(i_1\cdot i_2)
	\end{equation*}
	defines $x\in\mathbb{R}^{n_1\times n_2}$ from $y\in\mathbb{R}^{n_1n_2}$.
\end{definition}

\begin{definition}\label{def:productlike}
    Consider the tensors $x\in\mathbb{R}^{\boldsymbol{n}_1\times a\times\boldsymbol{n}_2}$ and $y\in\mathbb{R}^{\boldsymbol{n}_3\times b\times\boldsymbol{n}_4}$.
    Then the expression
    \begin{equation} \label{eq:outer}
        z(i^{\boldsymbol{n_1}},i^{\boldsymbol{n_2}},j_1,j_2,i^{\boldsymbol{n_3}},i^{\boldsymbol{n_4}}) = x(i^{\boldsymbol{n_1}},j_1,i^{\boldsymbol{n_2}})\cdot y(i^{\boldsymbol{n_3}},j_2,i^{\boldsymbol{n_4}})
    \end{equation}
    defines the tensor $z\in\mathbb{R}^{\boldsymbol{n}_1\times \boldsymbol{n}_2\times a\times b\times\boldsymbol{n}_3\times\boldsymbol{n}_4}$ in the obvious way.
    Similary, for $a=b$ the expression 
    \begin{equation} \label{eq:hadamard}
        z(i^{\boldsymbol{n_1}},i^{\boldsymbol{n_2}},j,i^{\boldsymbol{n_3}},i^{\boldsymbol{n_4}}) = x(i^{\boldsymbol{n_1}},j,i^{\boldsymbol{n_2}})\cdot y(i^{\boldsymbol{n_3}},j,i^{\boldsymbol{n_4}})
    \end{equation}
    defines the tensor $z\in\mathbb{R}^{\boldsymbol{n}_1\times \boldsymbol{n}_2\times a\times\boldsymbol{n}_3\times\boldsymbol{n}_4}$.
    Finally, also for $a=b$ the expression
    \begin{equation} \label{eq:inner}
        z(i^{\boldsymbol{n_1}},i^{\boldsymbol{n_2}},i^{\boldsymbol{n_3}},i^{\boldsymbol{n_4}}) = x(i^{\boldsymbol{n_1}},j,i^{\boldsymbol{n_2}})\cdot y(i^{\boldsymbol{n_3}},j,i^{\boldsymbol{n_4}})
    \end{equation}
    defines the tensor $z\in\mathbb{R}^{\boldsymbol{n}_1\times \boldsymbol{n}_2\times \boldsymbol{n}_3\times\boldsymbol{n}_4}$ as
    \begin{equation*}
        z(i^{\boldsymbol{n_1}},i^{\boldsymbol{n_2}},i^{\boldsymbol{n_3}},i^{\boldsymbol{n_4}}) = \sum_{k=1}^a x(i^{\boldsymbol{n_1}},k,i^{\boldsymbol{n_2}})\cdot y(i^{\boldsymbol{n_3}},k,i^{\boldsymbol{n_4}}) .
    \end{equation*}
\end{definition}
We choose this description mainly because of its simplicity and how it relates to the implementation of these operations in the numeric libraries \texttt{numpy}~\cite{oliphant_guide_2006} and \texttt{xerus}~\cite{huber_xerus_2014}.

\subsection{Graphical notation and tensor networks}
This section will introduce the concept of \textit{tensor networks}~\cite{espig_optimization_2011} and a graphical notation for certain operations which will simplify working with these structures.
To this end we reformulate the operations introduced in the last section in terms of nodes, edges and half edges.
\begin{definition}
    For a dimension tuple $\boldsymbol{n}$ of order $d$ and a tensor $x\in\mathbb{R}^{\boldsymbol{n}}$ the \textit{graphical representation} of $x$ is given by
    \begin{center}
    	\begin{tikzpicture}
    	\draw[black,fill=black] (0,0) circle (0.5ex);
    	\node[anchor=south east] at (0,0) {$x$};
    	\draw(-1,0)--(1,0);
    	\draw(0,1)--(0,-1);
    	\node[anchor=south] at (-1,0) {$i_1$};
    	\node[anchor=west] at (0,1) {$i_2$};	\node[anchor=north] at (1,0) {$i_3$};
    	\node at (0.5,-0.5) {\reflectbox{$\ddots$}};
    	\node[anchor=east] at (0,-1) {$i_d$};
    	\end{tikzpicture}
    \end{center}
    where the node represents the tensor and the half edges represent the $d$ different modes of the tensor illustrated by the symbolic indices $i_1,\ldots,i_d$.
\end{definition}
With this definition we can write the reshapings of Defintion~\ref{def:indexproduct} simply as
\begin{center}
	\begin{tikzpicture}
	\node[anchor=east] at (-1,0) {$x(i_1, i_2\cdot i_3\cdots i_d)\quad=\quad$};
	\draw[black,fill=black] (0,0) circle (0.5ex);
	\node[anchor=south, inner sep=5] at (0,0) {$x$};
	\draw(-1,0)--(1,0);
	\node[anchor=south] at (-1,0) {$i_1$};
	\node[anchor=south west] at (1,0) {$\hspace{-0.5em}i_2\cdot i_3\cdots i_d$};
	\end{tikzpicture}
\end{center}
and also simplify the binary operations of Definition~\ref{def:productlike}.
\begin{definition}
    Let $x\in\mathbb{R}^{\boldsymbol{n}_1\times a\times \boldsymbol{n}_2}$ and $y\in\mathbb{R}^{\boldsymbol{n}_3\times b\times \boldsymbol{n}_4}$ be two tensors. %
    Then Operation~\eqref{eq:outer} is represented by
    \begin{center}
    	\begin{tikzpicture}
    	\draw[black,fill=black] (0,0) circle (0.5ex);
    	\node[anchor=south east] at (0,0) {$x$};
    	\draw(-1,0)--(0,0);
    	\draw(0,1)--(0,-1);
    	\node[anchor=south] at (-1,0) {$i$};
    	\node[anchor=south] at (0,1) {{$i^{\boldsymbol{n_1}}$}};	
    	\node[anchor=north] at (0,-1) {{$i^{\boldsymbol{n_2}}$}};
    	
    	\draw[black,fill=black] (2,0) circle (0.5ex);
    	\node[anchor=south west] at (2,0) {$y$};
    	\draw(2,0)--(3,0);
    	\draw(2,1)--(2,-1);
    	\node[anchor=south] at (3,0) {$j$};
    	\node[anchor=south] at (2,1) {$i^{\boldsymbol{n_3}}$};	
    	\node[anchor=north] at (2,-1) {$i^{\boldsymbol{n_4}}$};
    	
    	\node at (4,0) {$=$};
    	
    	\draw[black,fill=black] (6,0) circle (0.5ex);
    	\node[anchor=south east] at (6,0) {$z$};
    	\draw(5,0)--(7,0);
    	\draw(6,1)--(6,-1);
    	\node[anchor=south] at (5,0) {$i$};
    	\node[anchor=south] at (7,0) {$j$};
    	\node[anchor=south] at (6,1) {{$i^{\boldsymbol{n_1}}\cdot i^{\boldsymbol{n_3}}$}};
    	\node[anchor=north] at (6,-1) {{$i^{\boldsymbol{n_2}}\cdot i^{\boldsymbol{n_4}}$}};
    	
    	\node[anchor=north] at (8,0) {.};
    	
    	\end{tikzpicture}
    \end{center}
    and defines $z\in\mathbb{R}^{\cdots\times a \times b \times\cdots}$.
    For $a=b$ Operation~\eqref{eq:hadamard} is represented by 
    \begin{center}
    	\begin{tikzpicture}
    	\draw[black,fill=black] (0,0) circle (0.5ex);
    	\node[anchor=south east] at (0,0) {$x$};
    	\draw(1,0.5)--(0,0);
    	\draw(0,1)--(0,-1);
    
        \node[anchor=south] at (0,1) {{$i^{\boldsymbol{n_1}}$}};	
    	\node[anchor=north] at (0,-1) {{$i^{\boldsymbol{n_2}}$}};
    	\node at (0.5,0) {$i$};
    	\draw[black,fill=black] (2,0) circle (0.5ex);
    	\node[anchor=south west] at (2,0) {$y$};
    	\draw(2,0)--(1,0.5);
    	\draw(2,1)--(2,-1);
    	\node[anchor=south] at (2,1) {$i^{\boldsymbol{n_3}}$};	
    	\node[anchor=north] at (2,-1) {$i^{\boldsymbol{n_4}}$};
    	
    	\node at (1.5,0) {$i$};
    	\draw(1,0.5)--(1,1);
    	
    	\node[anchor=south east] at (1,0.5) {$i$};
    	
    	\node at (4,0) {$=$};
    	
    	\draw[black,fill=black] (6,0) circle (0.5ex);
    	\node[anchor=south east] at (6,0) {$z$};
    	\draw(5,0)--(6,0);
    	\draw(6,1)--(6,-1);
    	\node[anchor=south] at (5,0) {$i$};
    
    	\node[anchor=south] at (6,1) {{$i^{\boldsymbol{n_1}}\cdot i^{\boldsymbol{n_3}}$}};
    	\node[anchor=north] at (6,-1) {{$i^{\boldsymbol{n_2}}\cdot i^{\boldsymbol{n_4}}$}};

    	\node[anchor=north] at (8,0) {.};
    	
    	\end{tikzpicture}
    \end{center}		
    and defines $z\in\mathbb{R}^{\cdots\times a  \times \cdots}$ and Operation~\eqref{eq:inner} defines $z\in\mathbb{R}^{\cdots \times \cdots}$ by
    \begin{center}
    	\begin{tikzpicture}
    	\draw[black,fill=black] (0,0) circle (0.5ex);
    	\node[anchor=south east] at (0,0) {$x$};
    	\draw(1,0)--(0,0);
    	\draw(0,1)--(0,-1);
    	
    	\node[anchor=south] at (0,1) {{$i^{\boldsymbol{n_1}}$}};	
    	\node[anchor=north] at (0,-1) {{$i^{\boldsymbol{n_2}}$}};
    	\node[anchor=north] at (0.5,0) {$i$};
    	\draw[black,fill=black] (2,0) circle (0.5ex);
    	\node[anchor=south west] at (2,0) {$y$};
    	\draw(2,0)--(1,0);
    	\draw(2,1)--(2,-1);
    	\node[anchor=south] at (2,1) {$i^{\boldsymbol{n_3}}$};	
    	\node[anchor=north] at (2,-1) {$i^{\boldsymbol{n_4}}$};
    	
    	\node[anchor=north] at (1.5,0) {$i$};
    	
    	\node at (4,0) {$=$};
    	
    	\draw[black,fill=black] (6,0) circle (0.5ex);
    	\node[anchor=south east] at (6,0) {$z$};
    	\draw(6,1)--(6,-1);
    	
    	\node[anchor=south] at (6,1) {{$i^{\boldsymbol{n_1}}\cdot i^{\boldsymbol{n_3}}$}};
    	\node[anchor=north] at (6,-1) {{$i^{\boldsymbol{n_2}}\cdot i^{\boldsymbol{n_4}}$}};

    	\node[anchor=north] at (8,0) {.};
    	
    	\end{tikzpicture}
    \end{center}
\end{definition}
With these definitions we can compose entire networks of multiple tensors which are called tensor networks.

\subsection{The Tensor Train Format}\label{subsec:tensortrains}

A prominent example of a tensor network is the \textit{tensor train (TT)}~\cite{oseledets_tensor-train_2011,holtz_alternating_2012}, which is the main tensor network used throughout this work.
This network is discussed in the following subsection.
\begin{definition}
    Let $\boldsymbol{n}$ be an dimensional tuple of order-$d$.
    The TT format decomposes an order $d$ tensor $x\in\mathbb{R}^{\boldsymbol{n}}$ into $d$ \emph{component tensors} $x_k\in\mathbb{R}^{r_{k-1}\times n_k\times r_k}$ for $k=1,\ldots,d$ with $r_0 = r_d = 1$.
    This can be written in tensor network formula notation as 
    \begin{equation*}
        x(i_1,\cdots,i_d) = x_1(i_1,j_1)\cdot x_2(j_1,i_2,j_2)\cdots x_d(j_{d-1},i_d).
    \end{equation*}
    The tuple $(r_1,\ldots,r_{d-1})$ is called the \emph{representation rank} of this representation.
\end{definition}
In graphical notation it looks like this
\begin{center}
	\begin{tikzpicture}
	\draw[black,fill=black] (0,0) circle (0.5ex);
	\node[anchor=south east] at (0,0) {$x$};
	\draw(-1,0)--(1,0);
	\draw(0,1)--(0,-1);
	\node[anchor=south] at (-1,0) {$i_1$};
	\node[anchor=west] at (0,1) {$i_2$};	\node[anchor=north] at (1,0) {$i_3$};
	\node at (0.5,-0.5) {\reflectbox{$\ddots$}};
	\node[anchor=east] at (0,-1) {$i_d$};
	
	\node[anchor=east] at (2,0) {$=$};

	\draw[black,fill=black] (3,0) circle (0.5ex);	
	\draw[black,fill=black] (4,0) circle (0.5ex);	
	\draw[black,fill=black] (5,0) circle (0.5ex);
	\node at (6,0) {\reflectbox{$\cdots$}};	
	\draw[black,fill=black] (7,0) circle (0.5ex);
	
	\draw(3,0)--(5.5,0);
	\draw(6.5,0)--(7,0);
	\draw(3,0)--(3,-0.75);
	\draw(4,0)--(4,-0.75);
	\draw(5,0)--(5,-0.75);
	\draw(7,0)--(7,-0.75);	
	
	\node[anchor=south] at (3,0) {$x_1$};
	\node[anchor=south] at (4,0) {$x_2$};	\node[anchor=south] at (5,0) {$x_3$};
	\node[anchor=south] at (7,0) {$x_d$};
	\node[anchor=north] at (3,-0.75) {$i_1$};
	\node[anchor=north] at (4,-0.75) {$i_2$};	\node[anchor=north] at (5,-0.75) {$i_3$};
	\node[anchor=north] at (7,-0.75) {$i_d$};
	\node[anchor=north] at (3.5,0) {$j_1$};
	\node[anchor=north] at (4.5,0) {$j_2$};	\node[anchor=north] at (5.5,0) {$j_3$};
	\node[anchor=north] at (6.5,0) {$j_{d-1}$};
	\end{tikzpicture}
\end{center}
\begin{remark}\label{rmk:TT_representation}
    Note that this representation is not unique.
    For any pair of matrices $(A,B)$ that satisfies $AB=\operatorname{Id}$ we can replace $x_k$ by $x_k(i_1,i_2,j)\cdot A(j,i_3)$ and $x_{k+1}$ by $B(i_1,j)\cdot x(j,i_2,i_3)$ without changing the tensor $x$.
\end{remark}
The representation rank of $x$ is therefore dependent on the specific representation of $x$ as a TT, hence the name.
Analogous to the concept of matrix rank we can define a minimal necessary rank that is required to represent a tensor $x$ in the TT format.
\begin{definition} 
	The \textit{tensor train rank} of a tensor $x\in\mathbb{R}^{\boldsymbol{n}}$ with tensor train components $x_1\in\mathbb{R}^{n_1\times r_1}$, $x_k\in\mathbb{R}^{r_{k-1}\times n_k\times r_k}$ for $k=2,\ldots,d-1$ and $x_d\in\mathbb{R}^{r_{d-1}\times n_d}$ is the set  
	\[
    	 \text{TT-rank}(x) = (r_1,\cdots,r_d)
	\]
	of minimal $r_k$'s such that the $x_k$ compose $x$.
\end{definition}
In \cite[Theorem 1a]{holtz_manifolds_2012} it is shown that the TT-rank can be computed by simple matrix operations.
Namely, $r_k$ can be computed by joining the first $k$ indices and the remaining $d-k$ indices and computing the rank of the resulting matrix.
At last, we need to introduce the concept of left and right orthogonality for the tensor train format. 
\begin{definition}
    Let $x\in\mathbb{R}^{\boldsymbol{m}\times n}$ be a tensor of order $d+1$.
    We call $x$ \emph{left orthogonal} if
    \[
        x(i^{\boldsymbol{m}},j_1) \cdot x(i^{\boldsymbol{m}},j_2)
        = \operatorname{Id}(j_1,j_2) .
    \]
    Similarly, we call a tensor $x\in\mathbb{R}^{m\times\boldsymbol{n}}$ of order $d+1$ \emph{right orthogonal} if 
    \[
        x(i_1, j^{\boldsymbol{n}}) \cdot x(i_2, j^{\boldsymbol{n}})
        = \operatorname{Id}(i_1,i_2) .
    \]
    A tensor train is \emph{left orthogonal} if all component tensors $x_1,\ldots,x_{d-1}$ are left orthogonal.
    It is \emph{right orthogonal} if all component tensors $x_2,\ldots,x_d$ are right orthogonal.
\end{definition}

\begin{lemma}[\cite{oseledets_tensor-train_2011}]
	For every tensor $x\in\mathbb{R}^{\boldsymbol{n}}$ of order $d$ we can find left and right orthogonal decompositions. 
\end{lemma}

For technical purposes it is also useful to define the so-called \textit{interface tensors}, which are based on left and right orthogonal decompositions.
\begin{definition}
	Let $x$ be a tensor train of order $d$ with rank tuple $\boldsymbol{r}$. %
	For every $k=1,\ldots,d$ and $\ell=1,\ldots,r_k$,
	the $\ell$-th \textit{left interface vector} is given by
	\begin{equation*}
    	\rmapleqq{k,\ell}{x}(i_1,i_2,\cdots,i_k)
    	= x_1(i_1,j_1)\cdots x_k(j_{k-1},i_k,\ell)
	\end{equation*}
	where $x$ is assumed to be left orthogonal.
	The $\ell$-th \textit{right interface vector} is given by
	\begin{equation*}
    	\rmapgeqq{k+1,\ell}{x}(i_{k+1},\cdots,i_d)
    	= x_{k+1}(\ell,i_{k+1},j_{k+1})\cdots x_d(j_{d-1},i_d)
	\end{equation*}
	where $x$ is assumed to be right orthogonal.
\end{definition}

\subsection{Sets of Polynomials} \label{subsec:sets}

In this section we specify the setup for our method and define the majority of the different sets of polynomials that are used.
We start by defining dictionaries of one dimensional functions which we then use to construct the different sets of high-dimensional functions.
\begin{definition}
	Let $p\in\mathbb{N}$ be given.
	A function dictionary of size $p$ is a vector valued function $\Psi:\mathbb{R} \rightarrow \mathbb{R}^{p}$.
\end{definition}
\begin{example}\label{ex:monomials}
	Two simple examples of a function dictionary that we use in this work are given by the monomial basis of dimension $p$, i.e. 
    \begin{equation}\label{eq:monomial}
    	\Psi_{\text{monomial}}(x) = \begin{pmatrix}1&x&x^2&\cdots&x^{p-1}\end{pmatrix}^T
	\end{equation}
	and by the basis of the first $p$ Legendre polynomials, i.e.
    \begin{equation}\label{eq:legendre}
    	\Psi_{\text{Legendre}}(x) = \begin{pmatrix}1&x&\frac 12 (3x^2-1)&\frac 12 (5x^3-3x)&\cdots\end{pmatrix}^T .
	\end{equation}
\end{example}
Using function dictionaries we can define the following high-dimensional space of multivariate functions.
	Let $\Psi$ be a function dictionary of size $p\in\mathbb{N}$.
	The $d$-th order product space that corresponds to the function dictionary $\Psi$ is 
    \begin{equation}\label{eq:ansatzspace}
        V_p^d := \left\langle\bigotimes_{k=1}^d \Psi_{m_k} \,:\, \boldsymbol{m}\in\mathbb{N}_p^d \right\rangle .
    \end{equation}
This means that every function $u\in V_p^d$ can be written as
\begin{equation} \label{eq:Vp_contraction}
    u(x_1,\ldots, x_d) =  c(i_1,\ldots,i_d) \prod_{k=1}^d \Psi(x_k)(i_k)
\end{equation}
with a coefficient tensor $c\in\mathbb{R}^{\boldsymbol{p}}$ where $\boldsymbol{p} = (p,\ldots,p)$ is a dimension tuple of order $d$.
Note that equation~\eqref{eq:Vp_contraction} uses the index notation from Definition~\ref{def:productlike} with arbitrary but fixed $x_k$'s.
Since $\mathbb{R}^{\boldsymbol{p}}$ is an intractably large space, it makes sense for numerical purposes to consider the subset
\begin{equation}\label{eq:TTspace}
    T_r(V_p^d) := \{ u\in V_p^d\,:\, \text{TT-rank}(c)\leq r\}
\end{equation}
where the TT rank of the coefficient is bounded.
Every $u\in T_r(V_p^d)$ can thus be represented graphically as
\begin{equation} \label{eq:Vp_contraction_TT}
    \begin{tikzpicture}[baseline=(current  bounding  box.center), inner sep=5]
	\draw[black,fill=black] (0.6,0) circle (0.5ex);
	\node[anchor=east] at (0.6,0) {$u(x_1,\ldots,x_d)$};
	\draw[black,fill=black] (3,0) circle (0.5ex);
	\draw[black,fill=black] (4,0) circle (0.5ex);
	\draw[black,fill=black] (5,0) circle (0.5ex);
	\draw[black,fill=black] (7,0) circle (0.5ex);
	
	\node[anchor=east] at (2,0) {$=$};
    
	\node at (6,0) {$\cdots$};
	\draw(3,-0)--(5.5,0);
	\draw(7,-0)--(6.5,0);
	\draw(3,-0.75)--(3,0);
	\draw(4,-0.75)--(4,0);
	\draw(5,-0.75)--(5,0);
	\draw(7,-0.75)--(7,0);
	\draw[black,fill=black] (3,-0.75) circle (0.5ex);
	\draw[black,fill=black] (4,-0.75) circle (0.5ex);	
	\draw[black,fill=black] (5,-0.75) circle (0.5ex);	
	\draw[black,fill=black] (7,-0.75) circle (0.5ex);	
	
	\node[anchor=south] at (3,0.) {$C_1$};
	\node[anchor=south] at (4,0.) {$C_2$};
	\node[anchor=south] at (5,0.) {$C_3$};
	\node[anchor=south] at (7,0.) {$C_d$};
    
	\node[anchor=north] at (3,-0.75) {$\Psi(x_1)$};
	\node[anchor=north] at (4,-0.75) {$\Psi(x_2)$};
	\node[anchor=north] at (5,-0.75) {$\Psi(x_3)$};
	\node[anchor=north] at (7,-0.75) {$\Psi(x_d)$};
	\node[anchor=north] at (7.5,-0.) {$.$};
	\end{tikzpicture}
\end{equation}
where the $C_k$'s are the components of the tensor train representation of the coefficient tensor $c\in\mathbb{R}^{\boldsymbol{p}}$ of $u\in V_p^d$.
\begin{remark}\label{rmk:tts_are_polynomials}
    In this way every tensor $c\in\mathbb{R}^{\boldsymbol{p}}$ (in the tensor train format) corresponds one to one to a function $u\in V_p^d$.
\end{remark}
An important subspace of $V_p^d$ is the space of homogeneous polynomials.
For the purpose of this paper we define the subspace of homogeneous polynomials of degree $g$ as the space
\begin{equation} \label{eq:homogeneousspace}
    W_g^d := \left\langle\bigotimes_{k=1}^d\Psi_{m_k} \,:\, \boldsymbol{m}\in\mathbb{N}_p^d \ \ \,\text{and}\ \,\sum_{k=1}^d m_k = d+g \right\rangle .
\end{equation}
From this definition it is easy to see that a homogeneous polynomial of degree $g$ can be represented as an element of $V_p^d$ where the coefficient tensor $c$ satisfies 
\[
    c(m_1,\ldots,m_d) = 0, \quad \sum_{k=1}^d m_k \neq d+g .
\]
In Section~\ref{sec:knownresults} we will introduce an efficient representation of such coefficient tensors $c$ in a block sparse tensor format.\\
Using $W_g^d$ we can also define the space of polynomials of degree at most $g$ by
\begin{equation}\label{eq:degbddpolyspace}
    S_{g}^d = \bigoplus_{\tilde{g}=0}^g W_{\tilde{g}}^d .
\end{equation}
Based on this characterization we will define a block-sparse tensor train version of this space in Section~\ref{sec:knownresults}.

\subsection{Parametrizing homogeneous polynomials by symmetric tensors}

In algebraic geometry the space $W_g^d$ is considered classically only for the dictionary $\Psi_{\mathrm{monomial}}$ of monomials and is typically parameterized by a symmetric tensor
\begin{equation} \label{eq:symmetric_polynomial}
    u(x) = B(i_1, \ldots, i_g) \cdot x(i_1) \cdots x(i_g) ,\quad x\in\mathbb{R}^d
\end{equation}
where $\boldsymbol{d} = (d,\ldots,d)$ is a dimension tuple  of order $g$ and $B\in\mathbb{R}^{\boldsymbol{d}}$ satisfies $B(m_1,\ldots,m_g) = B(\sigma(m_1,\ldots,m_g))$ for every permutation $\sigma$ in the symmetric group $S_g$.
We conclude this section by showing how the representation~\eqref{eq:Vp_contraction} can be calculated from the symmetric tensor representation~\eqref{eq:symmetric_polynomial}, and vice versa.
By equating coefficients we find that for every $(m_1,\ldots,m_d)\in\mathbb{N}_p^d$ either $m_1+\cdots+m_d \ne d+g$ and $c(m_1,\ldots,m_d) = 0$ or
\[
    c(m_1,\ldots,m_d) %
    = \sum_{\{\sigma(\boldsymbol{n})\,:\,\sigma\in S_g\}} B(\sigma(n_1,\ldots,n_g)) %
    \quad\text{where}\quad
    (n_1,\ldots,n_g)
    = (\underbrace{1, \ldots, 1}_{m_1-1\text{ times}}, \underbrace{2, \ldots, 2}_{m_2-1\text{ times}}, \ldots)
    \in \mathbb{N}_d^g . %
\]
Since $B$ is symmetric the sum simplifies to
\[
    \sum_{\{\sigma(\boldsymbol{n})\,:\,\sigma\in S_g\}} B(\sigma(n_1,\ldots,n_g)) %
    = \binom{g}{m_1-1,\ldots,m_d-1}B(n_1,\ldots,n_g).
\]
From this follows that for $(n_1,\ldots,n_g)\in\mathbb{N}_d^g$
\[
    B(n_1,\ldots,n_g) = \frac{1}{\binom{g}{m_1-1,\ldots,m_d-1}} c(m_1,\ldots,m_d)
    \quad\text{where}\quad
    m_k = 1+\sum_{\ell = 1}^g \delta_{k, n_\ell}
    \quad\text{for all}\quad k=1,\dots,d
\]
and $\delta_{k,\ell}$ denotes the \textit{Kronecker delta}.
This demonstrates how our approach can alleviate the difficulties that arise when symmetric tensors are represented in the hierarchical tucker format~\cite{hackbush_2016_symmetric} in a very simple fashion.

\subsection{Least Square}\label{subsec:leastsq}

Let in the following $V_p^d$ be the product space of a function dictionary $\Psi$ such that $V_p^d \subseteq L_2(\Omega)$.
Consider a high-dimensional function $f\in L_2(\Omega)$ on some domain $\Omega\subset \mathbb{R}^d$ and assume that the point-wise evaluation $f(x)$ is well-defined for $x\in\Omega$.
In practice it is often possible to choose $\Omega$ as a product domain $\Omega = \Omega_1\times\Omega_2\times\cdots\Omega_d$ by extending $f$ accordingly.
To find the best approximation $u_{W}$ of $f$ in the space $W\subseteq V_p^d$ we then need to solve the problem
\begin{align}\label{eq:L2minlsq}
	u_{W} = \argmin_{u\in W} \|f - u\|_{L_2(\Omega)}^2.
\end{align}
A practical problem that often arises when computing $u_{W}$ is that computing the $L_2(\Omega)$-norm is intractable for large $d$.
Instead of using classical quadrature rules one often resorts to a Monte Carlo estimation of the high-dimensional integral.
This means one draws $M$ random samples $\{x^{(m)}\}_{m=1,\ldots,M}$ from $\Omega$ and estimates
\begin{align*}
    \|f - u\|_{L_2(\Omega)}^2 \approx \frac{1}{M} \sum_{m = 1}^M\|f(x^{(m)}) - u(x^{(m)})\|_{\mathrm{F}}^2.
\end{align*}
With this approximation we can define an empirical version of $u_{W}$ as
\begin{equation}\label{eq:L2minlsq_empirical}
    u_{W,M} = \argmin_{u\in W} \frac{1}{M} \sum_{m = 1}^M\|f(x^{(m)}) - u(x^{(m)})\|_{\mathrm{F}}^2.
\end{equation}
For a linear space $W$, computing $u_{W,M}$ amounts to solving a linear system and does not pose an algorithmic problem.
We use the remainder of this section to comment on the minimization problem~\eqref{eq:L2minlsq_empirical} when a set of tensor trains is used instead. \\
Given samples $(x^{(m)})_{m = 1,\ldots,M}$ we can evaluate $u\in V_p^d$ for each $x^{(m)} = (x^{(m)}_{1},\ldots,x^{(m)}_{d})$ using equation~\eqref{eq:Vp_contraction}.
If the coefficient tensor $c$ of $u$ can be represented in the TT format then we can use equation~\eqref{eq:Vp_contraction_TT} to perform this evaluation efficiently for all samples $(x^{(m)})_{m=1,\ldots,M}$ at once.
For this we introduce for each $k=1,\ldots,d$ the matrix
\begin{equation}\label{eq:measurement_matrix}
    \Xi_k = \begin{pmatrix}
        \Psi(x^{(1)}_{k}) & \cdots & \Psi(x^{(M)}_{k})
    \end{pmatrix}\in\mathbb{R}^{p\times M}.
\end{equation}
Then the $M$-dimensional vector of evaluations of $u$ at all given sample points is given by 
\begin{center}
	\begin{tikzpicture}
		\begin{scope}[scale=0.75]
		\draw[black,fill=black] (0,0) circle (0.5ex);	
		\draw[black,fill=black] (0,1) circle (0.5ex);	
		\draw[black,fill=black] (0,-2) circle (0.5ex);	
	
		\draw[black,fill=black] (-1,0) circle (0.5ex);	
		\draw[black,fill=black] (-1,1) circle (0.5ex);	
		\draw[black,fill=black] (-1,-2) circle (0.5ex);	
	
		\draw (-4,-0.5)--(-3,-0.5)--(-1,1)--(0,1)--(0,-0.5);
		\draw (-3,-0.5)--(-1,0)--(0,0);
		\draw (-3,-0.5)--(-1,-2)--(0,-2)--(0,-1.5);
		
		\node at (-1,-1) {$\vdots$};
		\node at (0,-1) {$\vdots$};
		
		\node[anchor=west] at (0,1) {$C_1$};
		\node[anchor=west] at (0,0) {$C_2$};
		\node[anchor=west] at (0,-2) {$C_d$};
	
		\node[anchor=south east] at (-1,1) {$\Xi_1$};
		\node[anchor=south east] at (-1,0) {$\Xi_2$};
		\node[anchor=north east] at (-1,-2) {$\Xi_d$};
		\end{scope}
	\end{tikzpicture}
\end{center}
where we use Operation~\eqref{eq:hadamard} to join the different $M$-dimensional indices.
The alternating least-squares algorithm cyclically updates each component tensor $C_k$ by minimizing the residual corresponding to this contraction.
To formalize this we define the operator $\Phi_k\in\mathbb{R}^{M\times r_{k-1}\times n_k\times r_k}$ as
\begin{equation}\label{tikz:localcontraction}
\begin{tikzpicture}[baseline=(current  bounding  box.center)]	\begin{scope}[scale=0.75]
	\draw[black,fill=black] (0,0) circle (0.5ex);	
	\draw[black,fill=black] (0,-2) circle (0.5ex);	
		\draw[black,fill=black] (0,-4) circle (0.5ex);		\draw[black,fill=black] (0,-6) circle (0.5ex);	
	
	\draw[black,fill=black] (-1,0) circle (0.5ex);	
	\draw[black,fill=black] (-1,-2) circle (0.5ex);	
	\draw[black,fill=black] (-1,-3) circle (0.5ex);		\draw[black,fill=black] (-1,-4) circle (0.5ex);		\draw[black,fill=black] (-1,-6) circle (0.5ex);	
	
	\draw (-4,-3)--(-3,-3)--(-1,0)--(0,0)--(0,-0.5);
	\draw (-3,-3)--(-1,-2)--(0,-2);
	\draw (-3,-3)--(-1,-3)--(-0.5,-3);
	\draw (-3,-3)--(-1,-4)--(0,-4);
	\draw (-3,-3)--(-1,-6)--(0,-6)--(0,-5.5);
	\draw (0,-1.5)--(0,-2.5);
	\draw (0,-3.5)--(0,-4.5);
	
	\node at (-1,-1) {$\vdots$};
	\node at (0,-1) {$\vdots$};
	\node at (-1,-5) {$\vdots$};
	\node at (0,-5) {$\vdots$};
	
	\node[anchor=west] at (0,0) {$C_1$};
	\node[anchor=west] at (0,-2) {$C_{k-1}$};
	\node[anchor=west] at (0,-4) {$C_{k+1}$};
	\node[anchor=west] at (0,-6) {$C_d$};
	
	\node[anchor=south east] at (-1,0) {$\Xi_1$};
	\node[anchor=south east] at (-1,-2) {$\Xi_{k-1}$};
	\node[anchor=south east] at (-1,-3) {$\Xi_k$};
	\node[anchor=north east] at (-1,-4) {$\Xi_{k+1}$};
	\node[anchor=north east] at (-1,-6) {$\Xi_d$};
	
	\draw[dashed] (-2.5,-2.25) rectangle (1,0.55);
	\draw[dashed] (-2.5,-3.75) rectangle (1,-6.55);
	
	\node at (-4.5,-3) {$=$};
	
	\draw[black,fill=black] (-6,-3) circle (0.5ex);	
	\node[anchor=south east] at (-6,-3) {$\Phi_k$};
	\draw (-6.7,-3)--(-5.3,-3);
	\draw (-5.3,-2.25)--(-6,-3)--(-5.3,-3.75);
	\end{scope}
	\end{tikzpicture}
	.
\end{equation}	
Then the update for $C_k$ is given by a minimal residual solution of the linear system 
\[
    \Phi_k(j,i_1, i_2, i_3)\cdot C_k(i_1,i_2,i_3) = F(j)
\]
where $F(m) := y^{(m)} := f(x^{(m)})$ and  $i_1,i_2,i_3,j$ are symbolic indices of dimensions $r_{k-1},n_k,r_k,M$, respectively.
The particular algorithm that is used for this minimization may be adapted to the problem at hand.
These contractions are the basis for our algorithms in Section~\ref{sec:method}.
We refer to~\cite{holtz_alternating_2012} for more details on the ALS algorithm.
Note that it is possible to reuse parts of the contractions in $\Phi_k$ through so called \emph{stacks}.
In this way not the entire contraction has to be computed for every $k$.
The dashed boxes mark the parts of the contraction that can be reused.

\section{Theoretical Foundation}\label{sec:knownresults}

\subsection{Sample Complexity} \label{subsec:sample_complexity}

The quality of the solution $u_{W,M}$ of~\eqref{eq:L2minlsq_empirical} in relation to $u_{W}$ is subject to tremendous interest on the part of the mathematics community.
Two particular papers that consider this problem are~\cite{cohen_optimal_2017} and~\cite{eigel_convergence_2020}.
While the former provides sharper error bounds for the case of linear ansatz spaces the latter generalizes the work and is applicable to tensor network spaces. 
We now recall the relevant result for convenience.
\begin{proposition}[\cite{eigel_convergence_2020}] \label{prop:var_const}
    Define the \emph{variation constant}
    \begin{equation*}
        K(A) := \sup_{v\in A\setminus\{0\}} \frac{\|v\|_{L^\infty(\Omega)}^2}{\|v\|^2_{L^2(\Omega)}}.
    \end{equation*}
    Then for any $W$ with $k := \max\{K(\{f-u_{W}\}), K(\{u_{W}\} - W)\} < \infty$ it holds that
    \begin{equation*}
        \mathbb{P}\left[\|f - u_{W,M}\|_{L^2(\Omega)} \lesssim \|f - u_{W}\|_{L^2(\Omega)}\right] \ge 1 - q
    \end{equation*}
    where $q$ decreases exponentially with $\ln(q)\in \mathcal{O}(-n k^{-2})$.
\end{proposition}
Note that the value of $k$ depends only on $f$ and on the set $W$ but not on the particular choice of representation of $W$.
However, the variation constant of spaces like $V_p^d$ still depends on the underlying dictionary $\Psi$.
Although the proposition indicates that a low value of $k$ is necessary to achieve a fast convergence the tensor product spaces $V_p^d$ considered thus far does not exhibit a small variation constant.
The consequence of Proposition~\ref{prop:var_const} is that elements of this space are hard to learn in general and may require an infeasible number of samples.
To see this consider $\Omega = [-1,1]^d$ and the function dictionary $\Psi_{\mathrm{Legendre}}$ of Legendre polynomials~\eqref{eq:legendre} and let $\{P_{\boldsymbol{\ell}}\}_{\boldsymbol{\ell}\in L}$ be an orthonormal basis for some linear subspace $V\subseteq V_p^d$.
Then we can show that
\begin{equation} \label{eq:Kbound}
    K(V)
    = \sup_{x\in\Omega} \sum_{\boldsymbol{\ell}\in L} P_{\boldsymbol{\ell}}(x)^2
    = \sum_{\boldsymbol{\ell}\in L} \prod_{k=1}^d (2{\ell}_k+1) .
\end{equation}
by using techniques from~\cite[Section 3.1]{eigel_convergence_2020} and the fact that each $P_{\boldsymbol{\ell}}$ attains its maximum at $1$.
Using the product structure of $L=\mathbb{N}_p^d$ we can interchange the sum and product in~\eqref{eq:Kbound} and can conclude that $K(V_p^d) = p^{2d}$.
This means that we have to restrict the space $V_p^d$ to obtain an admissible variation constant.
We propose to use the space $W_g^d$ of homogeneous polynomials of degree $g$.
Spaces like this are commonly used in practical applications.
Their dimension is comparably low yet their direct sum $S_g^d$ allows for a good approximation of numerous highly regular functions given a sufficiently large polynomial degree $g$.
We can employ~\eqref{eq:Kbound} with $L = \{\boldsymbol{\ell} : |\boldsymbol{\ell}|=g\}$ to obtain the upper bound
\begin{align*}
    K(W_g^d)
    \le \binom{d-1+g}{d-1} \max_{|\boldsymbol{\ell}| = g} \prod_{k=1}^d (2{\ell}_k+1)
    \le \binom{d-1+g}{d-1} \left(2\left\lfloor \frac{g}{d}\right\rfloor + 3\right)^{g\bmod d}\left(2\left\lfloor \frac{g}{d}\right\rfloor + 1\right)^{d - g\bmod d}
\end{align*}
where the maximum is estimated by observing that $(2(\ell_1+1)+1)(2\ell_2+1) \le (2\ell_1+1)(2(\ell_2+1)+1) \Leftrightarrow \ell_2 \le \ell_1$.
For $g\le d$ this results in the simplified bound $K(W_g^d) \le (3\mathrm{e}\frac{d-1+g}{g})^g$.  %
This improves the variation constant substantially compared to the bound $K(V_p^d) \le p^{2d}$. \\
The bound for the dictionary of monomials $\Psi_{\mathrm{monomial}}$ is more involved but can theoretically be computed in the same way.
By drawing samples from an adapted sampling measure~\cite{cohen_optimal_2017} the theory in~\cite{eigel_convergence_2020} ensures that $K(V) = \operatorname{dim}(V)$ for all linear spaces $V$ --- independent of the underlying dictionary $\Psi$.
Using such an optimally weighted least-squares method thus leads to the bounds $K(V_p^d) = p^d$ and $K(W_g^d) = \binom{d-1+g}{d-1} \le (\mathrm{e}\frac{d-1+g}{g})^g$ for $g\le d$.

\subsection{Block Sparse Tensor Trains} \label{subsec:bstt}

Now that we have seen that it is advantagious to restrict ourselves to the space $W_g^d$ we need to find a way to do so without loosing the advantages of the tensor train format.
In~\cite{bachmayr} it was rediscovered that if a tensor train is an eigenvector of certain Laplace-like operators it admits a block sparse structure.
This means for a tensor train $c$ the components $C_k$ have zero blocks.
Furthermore, this block sparse structure is persevered under key operations, like e.g.\ the TT-SVD.
One possible operator which introduces such a structure is the Laplace-like operator 
\begin{align}\label{eq:degreeoperator}
    L = \sum_{k=1}^d  \biggl( \bigotimes_{\ell=1}^{k-1} I_{p} \biggr) \otimes \text{diag}(0,1,\ldots,p-1) \otimes \biggl( \bigotimes_{\ell=k+1}^{d} I_{p} \biggr).
\end{align}
This is the operator mentioned in the introduction encoding a quantum symmetry.
In the context of quantum mechanics this operator is known as the bosonic  particle number operator but we simply call it the degree operator.
The reason for this is that for the function dictionary of monomials $\Psi_{\text{monomial}}$ the eigenspaces of $L$ for eigenvalue $g$ are associated with homogeneous polynomials of degreee $g$.
Simply put, if the coefficient tensor $c$ for the multivariate polynomial $u\in V_p^d$ is an eigenvector of $L$ with eigenvalue $g$, then $u$ is homogeneous and the degree of $u$ is $g$.
In general there are polynomials in $V_p^d$ with degree up to $(p-1)d$.
To state the results on the block-sparse representation of the coefficient tensor we need the partial operators
\begin{align*}
    \opleq{L}{k} =& \sum_{m=1}^k  \biggl( \bigotimes_{\ell=1}^{m-1} I_{p} \biggr) \otimes \text{diag}(0,1,\ldots,p-1) \otimes \biggl( \bigotimes_{\ell=m+1}^{k} I_{p} \biggr)\\
    \opgeq{L}{k+1} =& \sum_{m=k+1}^d  \biggl( \bigotimes_{\ell=k+1}^{m-1} I_{p} \biggr) \otimes \text{diag}(0,1,\ldots,p-1) \otimes \biggl( \bigotimes_{\ell=m+1}^{d} I_{p} \biggr), 
\end{align*}
for which we have
\[
    L = \opleq{L}{k}\otimes \bigotimes_{\ell=k+1}^{d} I_{p} +  \bigotimes_{\ell=1}^{k} I_{p} \otimes \opgeq{L}{k+1}.
\]

In the following we %
adopt the notation $x=Lc$ to abbreviate the equation
\[
    x(i_1,\ldots,i_d) = L(i_1,\dots,i_d,j_1,\ldots,j_d)c(j_1,\ldots,j_d)
\]
where $L$ is a tensor operator acting on a tensor $c$ with result $x$.\\
Recall that by Remark~\ref{rmk:tts_are_polynomials} every TT corresponds to a polynomial by multiplying function dictionaries onto the cores.
This means that for every $\ell=1,\ldots,r$ the TT $\rmapleqq{k,\ell}{c}$ corresponds to a polynomial in the variables $x_1,\ldots,x_k$ and the TT $\rmapgeqq{k+1,\ell}{c}$ corresponds to a polynomial in the variables $x_{k+1},\ldots,x_d$.
In general these polynomials are not homogeneous, i.e.\ they are not eigenvectors of the degree operators $\opleq{L}{k}$ and $\opgeq{L}{k+1}$.
But since TTs are not uniquely defined (cf. Remark~\ref{rmk:TT_representation}) %
it is possible to find transformations of the component tensors $C_k$ and $C_{k+1}$ that do not change the tensor $c$ or the rank $r$ but result in a representation where each $\rmapleqq{k,\ell}{c}$ and each $\rmapgeqq{k+1,\ell}{c}$ correspond to a homogeneous polynomial.
Thus, if $c$ represents a homogeneous polynomial of degree $g$ and $\rmapleqq{k,\ell}{c}$ is homogeneous with $\operatorname{deg}(\rmapleqq{k,\ell}{c}) = \tilde{g}$ then $\rmapgeqq{k+1,\ell}{c}$ must be homogeneous with $\operatorname{deg}(\rmapgeqq{k,\ell}{c}) = g-\tilde{g}$.
This is put rigorously in the first assertion in the subsequent Theorem~\ref{thm:generalblockstructure}.
There $\mathcal{S}_{k,\tilde{g}}$ contains all the indices $\ell$ for which the reduced basis polynomials satisfy $\operatorname{deg}(\rmapleqq{k,\ell}{c}) = \tilde{g}$.
Equivalently, it groups the basis functions $\rmapgeqq{k+1,\ell}{c}$ into functions of order $g-\tilde{g}$.
The second assertion in Theorem~\ref{thm:generalblockstructure} states that we can only obtain a homogeneous polynomial of degree $\tilde{g}+m$ in the variables $x_1,\ldots,x_k$ by multiplying a homogeneous polynomial of degree $\tilde{g}$ in the variables $x_1,\ldots,x_{k-1}$ with a univariate polynomial of degree $m$ in the variable $x_k$.
This provides a constructive argument for the proof and can be used to ensure block-sparsity in the implementation.
Note that this condition forces entire blocks in the component tensor $C_k$ in equation~\eqref{eq:block_sparsity} to be zero and thus decreases the degrees of freedom.
\begin{theorem}\label{thm:generalblockstructure}\cite[Theorem 1]{bachmayr}
    Let $\boldsymbol{p} = (p,\ldots,p)$ be a dimension tuple of size $d$ and $c \in \mathbb{R}^{\boldsymbol{p}}\setminus\{0\}$, be a tensor train of rank $r = (r_1,\ldots, r_{d-1})$.
	Then $L c = g c$ if and only if $c$ has a representation with component tensors ${C_k\in\mathbb{R}^{r_{k-1}\times p\times r_k}}$ that satisfies the following two properties.
	\begin{enumerate}
	    \item For all $\tilde g \in \{0,1,\ldots,g\}$ there exist $\mathcal{S}_{k,\tilde g} \subseteq \{ 1,\ldots, r_k\}$ such that the left and right unfoldings satsify
    	\begin{equation}\label{evpi2}
    	\begin{aligned}
        	\opleq{L}{k} \rmapleqq{k,\ell}{c} &= \tilde{g} \rmapleqq{k,\ell}{c} \\
        	\opgeq{L}{k+1} \rmapgeqq{k+1,\ell}{c} &= (g-\tilde{g}) \rmapgeqq{k+1,\ell}{c}
    	\end{aligned}
    	\end{equation}
        for $\ell \in \mathcal{S}_{k,\tilde{g}}$.
    	\item The component tensors satisfy a block structure in the sets $\mathcal{S}_{k,\tilde g}$ for $m=1,\ldots p$
    	\begin{equation} \label{eq:block_sparsity}
    	    C_k(\ell_1,m,\ell_2) \ne 0  \quad\Rightarrow\quad \exists\ 0\le \tilde{g}\le g-(m-1)  :\ \ell_1\in\mathcal{S}_{k-1, \tilde{g}} \wedge \ell_2\in\mathcal{S}_{k,\tilde g+(m-1)}
    	\end{equation}
    	where we set $\mathcal{S}_{0,0} = \mathcal{S}_{d,g} = \{1\}$.
	\end{enumerate}
\end{theorem}
Note that this generalizes to other dictionaries and is not restricted to monomials.

\begin{remark}
    The rank bounds presented in this section do not only hold for the monomial dictionary $\Psi_{\text{monomial}}$ but for all polynomial dictionaries $\Psi$ that satisfy $\operatorname{deg}(\Psi_k) = k-1$ for all $k=1,\ldots,p$.
    When we speak of homogeneous polynomials of degree $g$ in the following we mean the space $W_g^d = \{v\in V_p^d : \operatorname{deg}(v) = g\}$.
    For the dictionary of monomials $\Psi_{\text{monomial}}$ the space ${W}_g^d$ contains only homogeneous polynomials in the classical sense.
    However, when the basis of Legendre polynomials $\Psi_{\text{Legendre}}$
	is used one obtains a space in which the functions are not homogeneous in the this sense.
	Note that we use polynomials since they have been applied successfully in practice, but other function dictionaries can be used as well.
	Also note that the theory is much more general as shown in~\cite{bachmayr} and is not restricted to the degree counting operator.
\end{remark}
Although, block sparsity also appears for $g+1\neq p$ we restrict ourselves to the case $g+1=p$ in this work.
Note that then the eigenspace of $L$ to the eigenvalue $g$ have the dimension equal to the space of homogeneous polynomials namely $\binom{d+g-1}{d-1}$ and for $\rho_{k,\tilde g} = |\mathcal{S}_{k,\tilde g}|$ we get the following rank bounds.
\begin{theorem}\label{thm:rankbound}\cite[Lemma 7]{bachmayr}
    Let $\boldsymbol{p} = (p,\ldots,p)$ be a dimension tuple of size $d$ and $c \in \mathbb{R}^{\boldsymbol{p}}\setminus\{0\}$, with $L c = g c$.
    Assume that $g + 1= p$ then the block sizes $\rho_{k,\tilde{g}}$ from Theorem~\ref{thm:generalblockstructure} are bounded by
    \begin{equation}\label{eq:rank_bounds}
        \rho_{k,\tilde g} \leq  \min \left\{\binom{k+\tilde g-1}{k-1},\binom{d-k+g-\tilde g-1}{d-k-1}\right\}
    \end{equation}
    for all $k=1,\ldots,d-1$ and $\tilde{g}=0,\ldots,g$ and $\rho_{k,0}=\rho_{k,g}=1$.
\end{theorem}
The proof of this theorem is based on a simple combinatorial argument.
For every $k$ consider the size of the groups $\rho_{k-1,\bar{g}}$ for $\bar{g}\le\tilde{g}$.
Then $\rho_{k,\bar{g}}$ can not exceed the sum of these sizes.
Similarly, $\rho_{k,\bar{g}}$ can not exceed $\sum_{\bar{g}\le\tilde{g}} \rho_{k+1,\bar{g}}$.
Solving these recurrence relations yields the bound.
\begin{example}[\textit{Block Sparsity}]
    Let $p=4$ and $g=3$ be given and let $c$ be a tensor train such that $Lc = gc$.
    Then for $k=2,\ldots,d-1$ the component tensors $C_k$ of $c$ exhibit the following block sparsity (up to permutation).
    For indices $i$ of order $r_{k-1}$ and $j$ of order $r_k$
    \[
    C_k(i,1,j) = \begin{pmatrix}* &0&0&0\\0&*&0&0\\0&0&*&0\\0&0&0&* \end{pmatrix}\, C_k(i,2,j) = \begin{pmatrix}0 &*&0&0\\0&0&*&0\\0&0&0&*\\0&0&0&0 \end{pmatrix}\,
    C_k(i,3,j) = \begin{pmatrix}0 &0&*&0\\0&0&0&*\\0&0&0&0\\0&0&0&0 \end{pmatrix}\,
    C_k(i,4,j) = \begin{pmatrix}0 &0&0&*\\0&0&0&0\\0&0&0&0\\0&0&0&0 \end{pmatrix}.
    \]
    This block structure results from sorting the indices $i$ and $j$ in such a way that $\max\mathcal{S}_{k,\tilde g}+1 = \min\mathcal{S}_{k,\tilde g+1}$ for every $\tilde g$.
    The maximal block sizes $\rho_{k,\tilde g}$ for $k = 1,\ldots, d-1$ are given by 
    \[
    \rho_{k, 0} = 1,
    \quad
    \rho_{k, 1} = \min\{k,d-k\},
    \quad
    \rho_{k, 2} = \min\{k,d-k\},
    \quad
    \rho_{k, 3} = 1. 
    \]
\end{example}
As one can see by Theorem~\ref{thm:rankbound} the block sizes $\rho_{k,\tilde g}$ can still be quite high.
The expressive power of tensor train parametrization can be understood by different concepts such as for example locality or self similarity.
For what comes now, we state a result that addresses locality and leads to $d$-independent rank bounds.
For this we need to introduce a workable notion of locality.
\begin{definition}
    Let $u\in W_g^d$ be a homogeneous polynomial and $B$ be the symmetric coefficient tensor %
    introduced in Subsection~\ref{subsec:sets}.
    We say that $u$ has a variable locality of $K_{\mathrm{loc}}$ if $B(\ell_1,\ldots,\ell_g)=0$ for all $(\ell_1,\ldots,\ell_g)\in\mathbb{N}_d^g$ with
    \[
        \max\{|\ell_{m_1}-\ell_{m_2}|\,:\,m_1,m_2=1,\ldots,g\} > K_{\mathrm{loc}} .
    \]
\end{definition}
\begin{example}
    Let $u$ be a homogeneous polynomial of degree $2$ with variable locality $K_{\mathrm{loc}}$.
    Then the symmetric matrix $B$ (cf.~\eqref{eq:symmetric_polynomial}) is $K_{\mathrm{loc}}$-banded.
    For $K_{\mathrm{loc}}=0$ this means that $B$ is diagonal and that $u$ takes the form
    \begin{equation*}
        u(x) = \sum_{\ell=1}^d B_{\ell\ell} x_\ell^2 .
    \end{equation*}
    This shows that variable locality removes mixed terms.
\end{example}
\begin{theorem}\label{thm:rankboundlocal}
    Let $\boldsymbol{p} = (p,\ldots,p)$ be a dimension tuple of size $d$ and $c \in \mathbb{R}^{\boldsymbol{p}}\setminus\{0\}$ correspond to a homogeneous polynomial of degree $g+1=p$ (i.e.\ $Lc = gc$) with variable locality $K_{\mathrm{loc}}$.
    Then the block sizes $\rho_{k,\tilde{g}}$ %
    are bounded by
    \begin{align}\label{eq:localrankbound}
        \rho_{k,\tilde g} \leq \sum_{\ell =1}^{K_\mathrm{loc}}\min\left\{\binom{K_\mathrm{loc}-\ell+1+\tilde g -2}{K_\mathrm{loc}-\ell},\binom{\ell+g-\tilde g -2}{\ell-1}\right\}
    \end{align}
    for all $k=1,\ldots,d-1$ and $\tilde{g} = 1,\ldots,g-1$
    as well as $\rho_{k,0}=\rho_{k,g}=1$.
\end{theorem}
\begin{proof}
    For fixed $g>0$ and a fixed component $C_k$ recall that for each $l$ the tensor $\tau_{k,l}^\le(c)$ corresponds to a reduced basis function $v_l$ in the variables $x_1,\ldots,x_k$ and that for each $l$ the tensor $\tau_{k+1,l}^\ge(c)$ corresponds to a reduced basis function $w_{l}$ in the variables $x_{k+1},\ldots,x_d$.
    Further recall that the sets $\mathcal{S}_{k,\tilde{g}}$ group these $v_l$ and $w_l$. 
    For all $l\in\mathcal{S}_{k,\tilde{g}}$ it holds that $\operatorname{deg}(v_l) = \tilde{g}$ and $\operatorname{deg}(w_l) = g-\tilde{g}$.
    For $\tilde g = 0$ and $\tilde g = g$ we know from Theorem~\ref{thm:rankbound} that $\rho_{k,\tilde{g}} = 1$.
    Now fix any $0 < \tilde{g} < g$ and arrange all the polynomials $v_l$ of degree $\tilde{g}$ in a vector $v$ and all polynomials $w_l$ of degree $g-\tilde{g}$ in a vector $w$.
    Then every polynomial of the form $v^\intercal Q w$ for some matrix $Q$ satisfies the degree constraint and the maximal possible rank of $Q$ provides an upper bound for the block size $\rho_{k,\tilde g}$. 
    However, due to the locality constraint we know that certain entries of $Q$ have to be zero.
    We denote a variable of a polynomial as inactive if the polynomial is constant with respect to changes in this variable and active otherwise.
    Assume that the polynomials in $v$ are ordered (ascendingly) according to the smallest index of their active variables and that the polynomials in $w$ are ordered (ascendingly) according to the largest index of their active variables.
    With this ordering $Q$ takes the form
    \begingroup
    \newcommand\bigzero{\makebox(0,0){\text{\huge0}}}
    \newcommand{\bord}[1]{\multicolumn{1}{c|}{#1}}
    \begin{equation*}
        Q = \begin{pmatrix} \vphantom{0} &            &           &        &                    & & \\
                            \vphantom{0} &            &           &        & \bigzero           & & \\ \cline{1-1}
                            \bord{Q_1}   &            &           &        &                    & & \\ \cline{2-2}
                            *            & \bord{Q_2} &           &        &                    & & \\ \cline{3-3}
                            *            & *          & \bord{Q_3}&        &                    & & \\
                            \vdots       & \vdots     & \vdots    & \ddots &                    & & \\ \cline{5-5}
                            *            & *          & *         & \cdots & \bord{Q_{K_\mathrm{loc}}} & \hphantom{0} & \hphantom{0}
            \end{pmatrix} .
    \end{equation*}
    \endgroup
    This means that for $\ell=1,\ldots,K_{\mathrm{loc}}$ each block $Q_\ell$ matches a polynomial $v_l$ of degree $\tilde{g}$ in the variables $x_{k-K_{\mathrm{loc}}+\ell},\ldots,x_k$ with a polynomial $w_l$ of degree $g-\tilde{g}$ in the variables $x_{k+1},\ldots,x_{k+\ell}$.
    Observe that the number of rows in $Q_\ell$ decreases while the number columns increases with $\ell$.
    This means that we can subdivide $Q$ as
    \begin{equation*}
        Q = \begin{pmatrix} 0   & 0   & 0 \\
                            Q_{\mathrm{C}} & 0   & 0 \\ 
                            *   & Q_{\mathrm{R}} & 0 
            \end{pmatrix},
    \end{equation*}
    where $Q_{\mathrm{C}}$ contains the blocks $Q_\ell$ with more rows than columns (i.e.\ full column rank) and $Q_{\mathrm{R}}$ contains the blocks $Q_\ell$ with more columns than rows (i.e.\ full row rank).
    So $Q_{\mathrm{C}}$ is a tall-and-skinny matrix while $Q_{\mathrm{R}}$ is a short-and-wide matrix and the rank for general $Q$ is bounded by the sum over the column sizes of the $Q_\ell$ in $Q_{\mathrm{C}}$ plus the sum over the row sizes of the $Q_\ell$ in $Q_{\mathrm{R}}$ i.e.
    \begin{equation*}
        \operatorname{rank}(Q) = \sum_{\ell=1}^{K_{\mathrm{loc}}} \operatorname{rank}(Q_\ell) .
    \end{equation*}
    To conclude the proof it remains to compute the row and column sizes of $Q_\ell$.
    Recall that the number of rows of $Q_\ell$ equals the number of polynomials $u$ of degree $\tilde{g}$ in the variables $x_{k-K_{\mathrm{loc}}+\ell},\ldots,x_{k}$ that can be represented as $u(x_{k-K_{\mathrm{loc}}+\ell},\ldots,x_{k}) = x_{k-K_{\mathrm{loc}}+\ell}\tilde{u}(x_{k-K_{\mathrm{loc}}+\ell},\ldots,x_{k})$.
    This corresponds to all possible $\tilde{u}$ of degree $\tilde{g}-1$ in the $K_{\mathrm{loc}}-\ell+1$ variables $x_{k-K_{\mathrm{loc}}+\ell},\ldots,x_{k}$.
    This means that
    \begin{equation*}
        \operatorname{\#rows}(Q_\ell) \le \binom{K_{\mathrm{loc}}-\ell+1+\tilde{g}-2}{K_{\mathrm{loc}}-\ell}
    \end{equation*}
    and a similar argument yields
    \begin{equation*}
        \operatorname{\#columns}(Q_\ell) \le \binom{\ell+g-\tilde{g}-2}{\ell-1} .
    \end{equation*}
    This concludes the proof.
\end{proof}
This lemma demonstrates how the combination of the model space $W_g^d$ with a tensor network space can improve the space complexity by incorporating locality.
\begin{remark}
    The rank bound in Theorem~\ref{thm:rankboundlocal} is only sharp for the highest possible rank.
    At the sides of the tensor trains the ranks can be much lower, but the precise bounds are quite technical to write down, which is why we skipped this.
    One sees that the bound only depends on $g$ and $K_\mathrm{loc}$ and is therefore $d$-independent.
\end{remark}
With Theorem~\ref{thm:rankboundlocal} it is possible to formulate situations in which a block sparse tensor train representation perform exceptionally well.
Let $u$ be a homogeneous polynomial with symmetric coefficient tensor $B$ (cf.~\eqref{eq:symmetric_polynomial}) and let $B|_{K_{\mathrm{loc}}}$ be the restriction of $B$ onto the coefficients that satisfy the variable locality constraint $K_{\mathrm{loc}}$. %
If we can choose $K_\mathrm{loc}$ such that the error of this restriction is small $u$ can be well approximated by a block sparse tensor train satisfying the rank bounds~\eqref{eq:localrankbound}.

\section{Method Description}\label{sec:method}
In this section we utilize the insights of Section~\ref{sec:knownresults} to refine the approximation spaces $W_g^d$ and $S_g^d$ and adapt the \emph{alternating least-squares (ALS)} method to solve the related least-squares problems.
First, we define the subset %
\begin{equation}
    B_\rho(W_g^d) := \{ u\in W_g^d\,:\, c \text{ is block-sparse with }\rho_{k,\tilde{g}}\le\rho \text{ for } 0\leq \tilde g \leq g\}
\end{equation}
and provide an algorithm for the related least-squares problem in Algorithm~\ref{alg:extended_ALS} which is a slightly modified version of the classical ALS~\cite{holtz_alternating_2012}\footnote{It is possible to include rank adaptivity as in SALSA~\cite{grasedyck_stable_2019} or bASD~\cite{eigel_non-intrusive_2019} and we have noted this in the relevant places.
}.
With this definition a straight-forward application of the concept of block-sparsity to the space $S_g^d$ is given by
\begin{equation}\label{eq:Sgrhod}
    S_{g,\rho}^d = \bigoplus_{\tilde{g}=0}^g B_\rho(W_{\tilde{g}}^d) .
\end{equation}
This means that every polynomial in $S_{g,\rho}^d$ can be represented by a sum of orthogonal coefficient tensors\footnote{The orthogonality comes from the symmetry of $L$ which results in orthogonal eigenspaces.}
\begin{equation}\label{eq:firstansatz}
    \sum_{\tilde{g}=0}^{g} c^{(\tilde{g})} \quad\text{where}\quad Lc^{(\tilde{g})}=\tilde{g}c^{(\tilde{g})} .
\end{equation}
There is however another, more compact, way to represent this function.
Instead of storing $g+1$ different tensors $c^{(0)},\ldots,c^{(g)}$ of order $d$, we can merge them into a single tensor $c$ of order $d+1$ such that $c(i^d, \tilde{g}) = c^{(\tilde{g})}(i^d)$.
The summation over $\tilde{g}$ can then be represented by a contraction of a vector of $1$'s to the $(d+1)$-th mode.
To retain the block-sparse representation we can view the $(d+1)$-th component as an artificial component representing a shadow variable $x_{d+1}$.
\begin{remark}
    The introduction of the shadow variable $x_{d+1}$ contradicts the locality assumptions of Theorem~\ref{thm:rankboundlocal} and implies that the worst case rank bounds must increase.
    This can be problematic since the block size contributes quadratically to the number of parameters.
    However, a similar argument as in the proof of Theorem~\ref{thm:rankboundlocal} can be made in this setting and one can show that the bounds remain independent of $d$
    \begin{align}\label{eq:localrankboundext}
        \rho_{k,\tilde g} \leq \underline{1+} \sum_{\ell =1}^{K_\mathrm{loc}}\min\left\{\binom{K_\mathrm{loc}-\ell+1+\tilde g -2}{K_\mathrm{loc}-\ell},\binom{\ell\underline{+1}+g-\tilde g -2}{\ell\underline{+1}-1}\right\}
    \end{align}
    where the changes to \eqref{eq:localrankbound} are underlined.
\end{remark}
We denote the set of polynomials that results from this augmented block-sparse tensor train representation as
\begin{equation}
    S_{g,\rho}^{d,\mathrm{aug}}
\end{equation}
where again $\rho$ provides a bound for the block-size in the representation.\\
Since $S_{g,\rho}^{d,\mathrm{aug}}$ is defined analogously to $B_\rho(W_g^d)$ we can use Algorithm~\ref{alg:extended_ALS} to solve the related least-squares problem by changing the contraction~\eqref{tikz:localcontraction} to
\begin{equation}\label{tikz:localcontractionmodified}
\begin{tikzpicture}[baseline=(current  bounding  box.center)]
    \begin{scope}[scale=0.75]
    \draw[black,fill=black] (0,0) circle (0.5ex);	
	\draw[black,fill=black] (0,-2) circle (0.5ex);	
		\draw[black,fill=black] (0,-4) circle (0.5ex);		\draw[black,fill=black] (0,-6) circle (0.5ex);	
	
	\draw[black,fill=black] (-1,0) circle (0.5ex);	
	\draw[black,fill=black] (-1,-2) circle (0.5ex);	
	\draw[black,fill=black] (-1,-3) circle (0.5ex);		\draw[black,fill=black] (-1,-4) circle (0.5ex);		\draw[black,fill=black] (-1,-6) circle (0.5ex);	
	\draw[black,fill=black] (-1,-7) circle (0.5ex);	
	\draw[black,fill=black] (0,-7) circle (0.5ex);	
	\draw (-4,-3)--(-3,-3)--(-1,0)--(0,0)--(0,-0.5);
	\draw (-3,-3)--(-1,-2)--(0,-2);
	\draw (-3,-3)--(-1,-3)--(-0.5,-3);
	\draw (-3,-3)--(-1,-4)--(0,-4);
	\draw (-3,-3)--(-1,-6)--(0,-6)--(0,-5.5);
	\draw (0,-1.5)--(0,-2.5);
	\draw (0,-3.5)--(0,-4.5);
	\draw (-3,-3)--(-1,-7)--(0,-7)--(0,-6);
	
	\node at (-1,-1) {$\vdots$};
	\node at (0,-1) {$\vdots$};
	\node at (-1,-5) {$\vdots$};
	\node at (0,-5) {$\vdots$};
	
	\node[anchor=west] at (0,0) {$C_1$};
	\node[anchor=west] at (0,-2) {$C_{k-1}$};
	\node[anchor=west] at (0,-4) {$C_{k+1}$};
	\node[anchor=west] at (0,-6) {$C_d$};
	\node[anchor=west] at (0,-7) {Id};
	
	\node[anchor=south east] at (-1,0) {$\Xi_1$};
	\node[anchor=south east] at (-1,-2) {$\Xi_{k-1}$};
	\node[anchor=south east] at (-1,-3) {$\Xi_k$};
	\node[anchor=north east] at (-1,-4) {$\Xi_{k+1}$};
	\node[anchor=north east] at (-1,-6) {$\Xi_d$};
	\node[anchor=north] at (-1,-7) {$\mathbf{1}\in\mathbb{R}^{M,g+1}$};
	
	\node at (-4.5,-3) {$=$};
	
	\draw[black,fill=black] (-6,-3) circle (0.5ex);	
	\node[anchor=south east] at (-6,-3) {$\Phi_k$};
	\draw (-6.7,-3)--(-5.3,-3);
	\draw (-5.3,-2.25)--(-6,-3)--(-5.3,-3.75);
	\end{scope}
	\end{tikzpicture}
	.
\end{equation}
To optimize the coefficient tensors $c^{(0)},\ldots,c^{(g)}$ in the space $S_{g,\rho}^d$ we resort to an alternating scheme.
Since the coefficient tensors are mutually orthogonal we propose to optimize each $c^{(\tilde{g})}$ individually while keeping the other summands $\{c^{(k)}\}_{k\ne\tilde{g}}$ fixed.
This means that we solve the problem
\begin{align} \label{eq:L2minlsqrestricted}
    u^{(\tilde{g})} = \argmin_{u\in W_{\tilde{g}}^d} \frac{1}{M} \sum_{m=1}^M \|f(x^{(m)}) - \sum_{\substack{k=0\\k\ne\tilde{g}}}^g u^{(k)}(x^{(m)}) - u(x^{(m)})\|_{\mathrm{F}}^2
\end{align}
which can be solved using Algorithm~\ref{alg:extended_ALS}.
The original problem~\eqref{eq:L2minlsq_empirical} is then solved by alternating over $\tilde{g}$ until a suitable convergence criterion is met.
The complete algorithm is summarized in Algorithm~\ref{alg:alternating_extended_ALS}.\\
The proposed representation has several advantages.
The optimization with the tensor train structure is computationally less demanding than solving directly in $S_{g}^d$.
Let $D = \operatorname{dim}(S_g^d) = \binom{d+g}{d}$.
Then a reconstruction on $S_g^d$ requires to solve a linear system of size $M\times D$ while a microstep in an ALS sweep only requires the solution of systems of size less than $Mpr^2$ (depending on the block size).
Moreover, the stack contractions as shown in~\ref{subsec:leastsq} also benefit from the block sparse structure.
This also means that the number of parameters of a full rank $r$ tensor train can be much higher than the number of parameters of several $c^{(m)}$'s which individually have ranks that are even larger than $r$.
\begin{remark}
    We expect that solving the least-squares problem for $S_{g,\rho}^{d,\mathrm{aug}}$ will be faster than for $S_{g,\rho}^{d}$ since it is computational more efficient to optimize all polynomials simultaneously than every degree individually in an alternating fashion.
    On the other hand, the hierarchical scheme of the summation approach may allow one to utilize multi-level Monte Carlo approaches.
    Together with the fact that every degree $\tilde{g}$ possesses a different optimal sampling density this may result in a drastically improved best case sample efficiency for the direct method.
    Additionally, with $S_{g,\rho}^d$ it is easy to extend the ansatz space simply by increasing $g$ which is not so straight-forward for $S_{g,\rho}^{d,\mathrm{aug}}$.
    Which approach is superior depends on the problem at hand.
\end{remark}

\begin{algorithm}[t!]
\SetKwInOut{Input}{input}\SetKwInOut{Output}{output}
\SetKwInOut{Input}{input} 
\Input{Data pairs $(x^{(m)},y^{(m)})\in\mathbb{R}^d\times \mathbb{R}$ for $m = 1,\ldots,M$, a function dictionary $\Psi$, a maximal degree $g$, and a maximal block size $\rho$.}
\Output{Coefficent tensor $c$ of a function $u\in B(W_g^d)$ that approximates the data.}
\BlankLine
For $k=1,\ldots,d$ compute $\Xi_k$ according to Equation~\eqref{eq:measurement_matrix}\;
Initialize the coefficient tensor $c$ for $u\in B(W_g^d)$\;
\textcolor{gray}{Initialize SALSA parameters\;}
\While{not converged}{
    Right orthogonalize $c$\;
    \For{$k=1,\ldots,d$}{
        Compute $\Phi_k$ according to Equation~\eqref{tikz:localcontraction}\;
        Compute the index set $\mathcal{I}$ of the non-zeros components in $C_k$ according to Equation~\eqref{eq:block_sparsity}\;
        Update $C_k$ by solving \textcolor{gray}{the SALSA-regularized version of} $\Phi_k(j,i^3)\cdot C_k(i^3) = y(j)$ restricted to $i^3\in\mathcal{I}$\;
        Left orthogonalize $C_k$ \textcolor{gray}{and adapt the $k$\textsuperscript{th} rank while respecting block size bounds $\rho$ and~\eqref{eq:rank_bounds}\;}
    }
    \textcolor{gray}{Update SALSA parameters\;}
}
\Return{ $c$ }
\caption{Extended ALS \textcolor{gray}{(SALSA)} for the least-squares problem on $B_\rho(W_g^d)$
}\label{alg:extended_ALS}
\end{algorithm}
\begin{algorithm}[t!]
\SetKwInOut{Input}{input}\SetKwInOut{Output}{output}
\SetKwInOut{Input}{input} 
\Input{Data pairs $(x^{(m)}, y^{(m)})\in\mathbb{R}^d\times \mathbb{R}$ for $m = 1,\ldots,M$, a function dictionary $\Psi$, a maximal degree $g$, and a maximal block size $\rho$.}
\Output{Coefficent tensors $c^{(0)},\ldots,c^{(g)}$ of a function $u\in S_{g,\rho}^d$ that approximates the data.}
\BlankLine
Initialize the coefficient tensors $c^{(\tilde{g})}$ of $u^{(\tilde{g})}\in B_\rho(W_{\tilde{g}}^d)$ for $\tilde{g}=0,\ldots,g$\;
\While{not converged}{
    \For{$\tilde{g} = 0,\ldots,g$}{
        Compute $z^{(m)} := y^{(m)} - \sum_{k\ne\tilde{g}} u^{(k)}(x^{(m)})$ for $m=1,\ldots,M$\;
        Update $c^{(\tilde{g})}$ by using Algorithm~\ref{alg:extended_ALS} on the data pairs $(x^{(m)}, z^{(m)})$ for $m=1,\ldots,M$\;
    }
}
\Return{ $c^{(\tilde{g})}$ for $\tilde{g}=0,\ldots,g$ }
\caption{Alternating extended ALS \textcolor{gray}{(SALSA)} for the least-squares problem on $S_{g,\rho}^d$}\label{alg:alternating_extended_ALS}
\end{algorithm}

\section{Numerical Results}\label{sec:numerics}

In this section we illustrate the numerical viability of the proposed framework on some simple but common problems.
We estimate the relative errors  on  test sets with respect to the sought function $f$.
Our implementation is meant only as a proof of concept and does not lay any emphasis on efficiency.
The termination conditions and the rank selection in particularly are na\"ively implemented and rank adaptivity is missing all together. It is, however, straight forward to apply SALSA as described in Section~\ref{sec:method} for rank adaptivity, which we consider to be state of the art for these kinds of problem. But, for our experiments, we are more interested in the required sample sizes leading to recovery.\\
In the following we  always assume $p=g+1$.
We also restrict the group sizes to be bounded by the parameter $\rho_{\mathrm{max}}$.
For every sample size the error plots show the distribution of the errors between the $0.15$ and $0.85$ quantile.
The code for all experiments has been made publicly available at \url{https://github.com/ptrunschke/block_sparse_tt}.

\subsection{Riccati equation}

In this section we consider the closed-loop linear quadratic optimal control problem
\begin{equation*}
\begin{aligned}
    &\underset{u}{\text{minimize}}
    && \|y\|_{L^2([0,\infty]\times[-1,1])}^2 + \lambda \|u\|_{L^2([0,\infty])}^2 \\
    &\text{subject to}
    && \partial_t y = \partial_x^2 y + u(t)\chi_{[-0.4,0.4]}, \; (t,x)\in [0,\infty]\times [-1,1] \\
    &&& y(0,x) = y_0(x), \; x\in [-1,1] \\
    &&& \partial_x y(t,-1) = \partial_x y(t,1) = 0
\end{aligned}
\end{equation*}
After a spatial discretization of the heat equation with finite differences we obtain a $d$-dimensional system of the form
\begin{equation*}
    \underset{u}{\text{minimize}}\ \int_0^\infty \boldsymbol{y}(t)^\intercal Q \boldsymbol{y}(t) + \lambda u(t)^2 \,\mathrm{d}t \quad\text{subject to}\quad  \dot{\boldsymbol{y}} = A \boldsymbol{y} + B u \quad\text{and}\quad \boldsymbol{y}(0) = \boldsymbol{y}_0 .
\end{equation*}
It is well known~\cite{Curtain_1995} that the value function for this problem takes the form $v(\boldsymbol{y}_0) = \boldsymbol{y}_0^{\intercal}P\boldsymbol{y}_0$ where $P$ can be computed by solving the \emph{algebraic Riccati equation (ARE)}.
It is therefore a homogeneous polynomial of degree $2$. This function is a perfect example of a function that can be well-approximated in the space ${W}_2^d$.
We approximate the value function on the domain $\Omega = [-1,1]^d$ for $d=8$ with the parameters $g=2$ and $\rho_{\mathrm{max}}=4$.

In this experiment we use the dictionary of monomials $\Psi = \Psi_{\mathrm{monomial}}$ (cf. equation~\eqref{eq:monomial}) and compare the ansatz spaces ${W}_2^8$, $B_4(W_2^8)$, $T_6(V_3^8)$ and $V_3^8$.
Since the function $v(x)$ is a general polynomial %
we use Theorem~\ref{thm:rankbound} to calculate the maximal block size $4$.
This guarantees perfect reconstruction since $B_4(W_{2}^8) = {W}_2^8$.
The rank bound $6$ is chosen s.t.\ $B_4(W_{2}^8) \subseteq T_6(V_3^8)$.
The degrees of freedom of all used spaces are listed in Table~\ref{tbl:riccati}.
In Figure~\ref{fig:riccati} we compare the relative error of the respective ansatz spaces.
It can be seen that the block sparse ansatz space recovers almost as well as the sparse approach.
As expected, the dense TT format is less favorable with respect to the sample size.

A clever change of basis, given by the diagonalization of $Q$, can reduce the required block size from $4$ to $1$.
This allows to extend the presented approach to higher dimensional problems.
The advantage over the classical Riccati approach becomes clear when considering 
non-linear versions of the control problem that do not exhibit a Riccati solution.
This is done in~\cite{oster_approximating_2020,dolgov_tensor_2021} using the dense TT-format $T_r(V_p^d)$.

\begin{table}[!ht]
    \centering
    \bgroup
    \def\arraystretch{1.25}  %
    \begin{tabular}{ | l | l | l | l | }
        \hline
        \multicolumn{1}{ | c | }{${W}_2^8$} & \multicolumn{1}{ c | }{$B_4(W_{2}^8)$} & \multicolumn{1}{ c | }{$T_6(V_{3}^8)$}  & \multicolumn{1}{ c | }{$V_3^8$} \\
        \hline
        $36$ & $94$ & $390$ & $6561$ \\
        \hline
    \end{tabular}
    \egroup
    \caption{Degrees of freedom for the full space $W_{g}^d$ of homogeneous polynomials of degree $g=2$, the TT variant $B_{\rho_{\rm max}}(W_g^d)$ with maximal block size $\rho_{\mathrm{max}}=4$, the space $T_r(V_p^d)$ with TT rank bounded by $r=6$, and the full space $V_p^d$ for completeness.}
    \label{tbl:riccati}
\end{table}

\begin{figure}[!ht]
    \centering
    \includegraphics[width=0.75\textwidth]{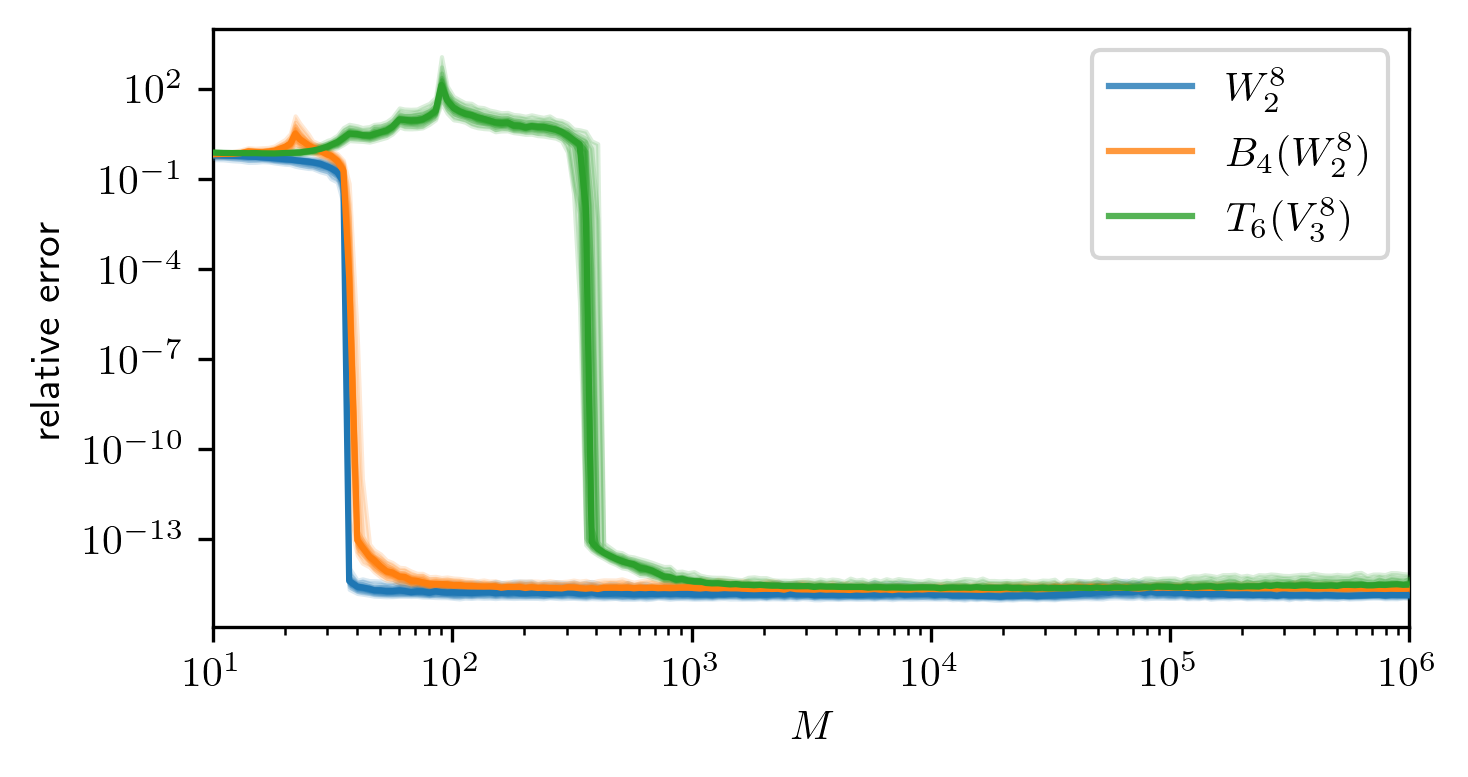}
    \caption{$0.15$--$0.85$ quantiles for recovery in blue: $W_{2}^8$, orange: $B_4( W_{2}^{8})$, and green: $T_6(V_3^8)$.}
    \label{fig:riccati}
\end{figure}

\subsection{Gaussian density}

As a second example we consider the reconstruction of an unnormalized Gaussian density
\begin{equation*}
    f(x) = \exp(-\|x\|_2^2) .
\end{equation*}
again on the domain $\Omega = [-1,1]^d$ with $d=6$.
For the dictionary $\Psi = \Psi_{\mathrm{Legendre}}$ (cf. equation~\eqref{eq:legendre}) we chose $g=7$, $\rho_{\mathrm{max}}=1$ and $r=8$ and compare the reconstruction w.r.t.\ $S_{g}^d$, $S_{g,\rho_{\mathrm{max}}}^d$ and $T_r(V_{p}^d)$, defined in \eqref{eq:degbddpolyspace}, \eqref{eq:Sgrhod} and \eqref{eq:TTspace}.
The degrees of freedom resulting from these different discretizations are compared in Table~\ref{tbl:gaussian}.
This example is interesting because here the roles of the spaces are reversed.
The function has product structure 
\[
f(x) = \exp(-x_1^2)\cdots\exp(-x_d^2)
\]
and can therefore be well approximated as a rank 1 tensor train with each component $C_k$ just being  a best approximation for $\exp(-x_k^2)$ in the used function dictionary. Therefore, we expect the higher degree polynomials to be important. 
A comparison of the relative errors to the exact solution are depicted in Figure~\ref{fig:gaussian}.
This example demonstrates the limitations of the ansatz space $S_7^6$ which is not able to exploit the low-rank structure of the function $f$.
Using $S_{7,1}^6$ can partially remedy this problem as can be seen by the improved sample efficiency.
But since $S_{7,1}^6\subseteq S_7^6$ the final approximation error of $S_{7,1}^6$ can not deceed that of $S_7^6$.
One can see that the dense format $T_1(V_8^6)$ produces the best results but is quite unstable compared to the other ansatz classes.
This instability is a result of the non-convexity of the set $T_r(V_p^d)$ and we observe that the chance of getting stuck in a local minimum increases when the rank $r$ is reduced from $8$ to $1$.
Finally, we want to address the peaks that are observable at $M\approx 500$ samples for $T_{8}(V_8^6)$ and $M\approx 1716$ samples for $S_7^6$.
For this recall that the approximation in $S_7^6$ amounts to solving a linear system which is underdetermined for $M<1716$ samples and overdetermined for $M>1716$ samples.
In the underdetermined case we compute the minimum norm solution and in the overdetermined case we compute the least-squares solution.
It is well-known that the solution to such a reconstruction problem is particularly unstable in the area of this transition~\cite{cohen_optimal_2017}.
Although the set $S_{7,1}^6$ is non-linear we take the peak at $M\approx 500$ as evidence for a similar effect which is produced by the similar linear systems that are solved in the micro steps in the ALS.

\begin{table}[!ht]
    \centering
    \bgroup
    \def\arraystretch{1.25}  %
    \begin{tabular}{ | l | l | l | l | l| }
        \hline
        \multicolumn{1}{ | c | }{$S_{7}^6$} & \multicolumn{1}{ c | }{$S_{7,1}^6$} & \multicolumn{1}{ c | }{$T_1(V_{8}^6)$}  & \multicolumn{1}{ c | }{$T_8(V_{8}^6)$} & \multicolumn{1}{ c | }{$V_8^6$} \\
        \hline
        $1716$ & $552$ & $48$ & $2176$ & $262144$ \\
        \hline
    \end{tabular}
    \egroup
    \caption{Degrees of freedom for the full space $S_{g}^d$, the TT variant  $S_{g,\rho_{\mathrm{max}}}^d$ with maximal block size $\rho_{\mathrm{max}}=1$, the space $T_r(V_p^d)$ with TT rank bounded by $r=1$, the space $T_r(V_p^d)$ with TT rank bounded by $r=8$, and the full space $V_p^d$ for completeness.}
    \label{tbl:gaussian}
\end{table}

\begin{figure}[!ht]
    \centering
    \includegraphics[width=0.75\textwidth]{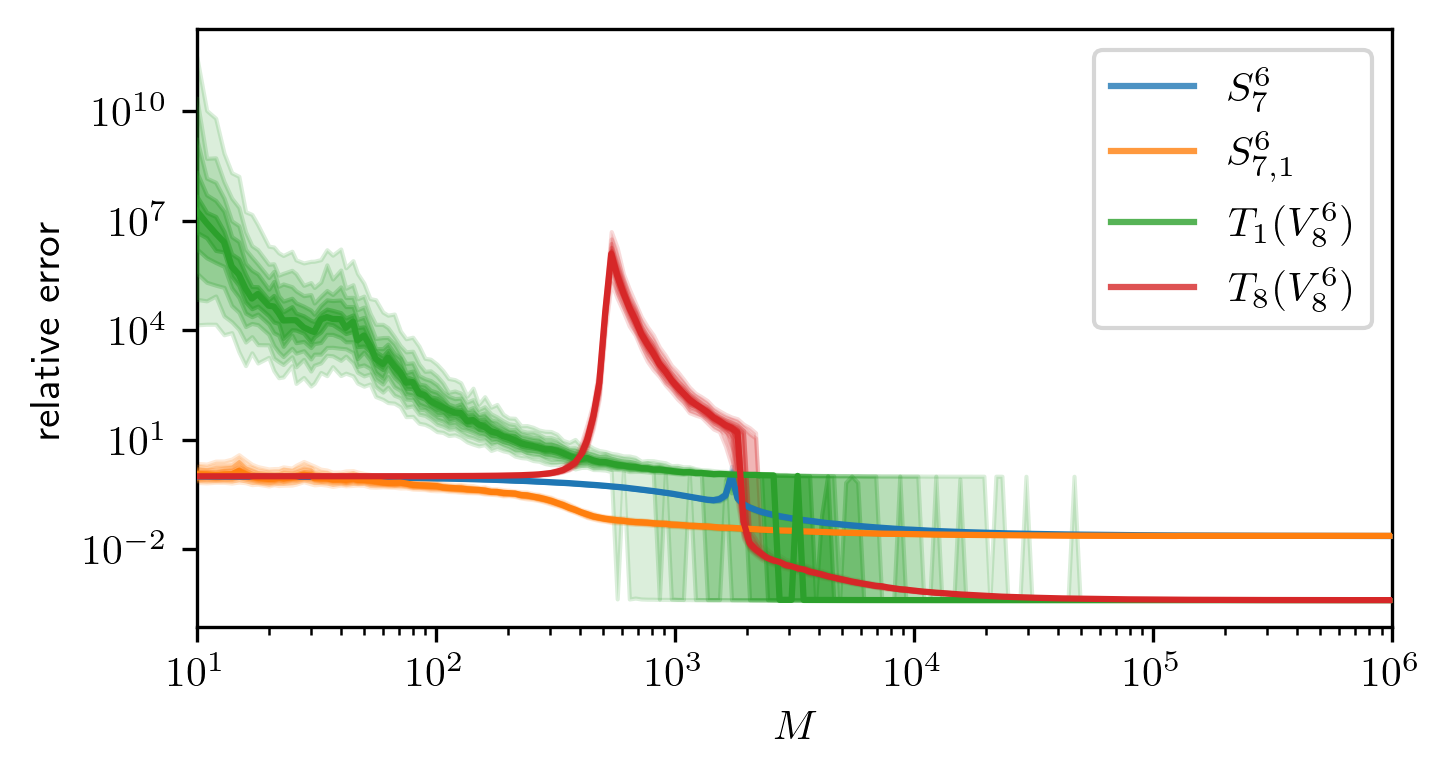}
    \caption{$0.15$--$0.85$ quantiles for recovery in blue: $S_{7}^6$, orange: $S_{7,1}^6$, green: $T_1(V_{8}^6)$, and red: $T_8(V_{8}^6)$.}    \label{fig:gaussian}
\end{figure}

\subsection{Quantities of Interest}

The next considered problem often arises when computing quantities of interest from random partial differential equations.
We consider the stationary diffusion equation
\begin{equation*}
    \begin{aligned}
        \nabla_x a(x,y) \nabla_x u(x,y) &= f(x)  &  x&\in D \\
        u(x,y) &= 0                              &  x&\in\partial D
    \end{aligned}
\end{equation*}
on $D = [-1,1]^2$.
This equation is parametric in $y\in[-1,1]^d$.
The randomness is introduced by the uniformly distributed random variable $y\sim \mathcal{U}([-1,1]^d)$ that enters the diffusion coefficient
\begin{equation*}
    a(x,y) := 1 + \frac{6}{\pi^2} \sum_{k=1}^d k^{-2} \sin(\hat{\varpi}_{k} x_1) \sin(\check{\varpi}_{k} x_2) y_k
\end{equation*}
with $\hat{\varpi}_k = \pi\lfloor\frac{k}{2}\rfloor$ and $\check{\varpi}_k = \pi\lceil\frac{k}{2}\rceil$.
The solution $u$ often measures the concentration of some substance in the domain $\Omega$ and one is  interested in the total amount of this substance in the entire domain
\begin{equation*}
    M(y) := \int_{\Omega} u(x,y) \,\mathrm{d}x .
\end{equation*}
An important result proven in~\cite{Hansen_2012} ensures the $\ell^p$ summability, for some $0<p\le1$, of the polynomial coefficients of the solution of this equation when $\Psi$ is the dictionary of Chebyshev polynomials.
This means that the function is very regular and we presume that it can be well approximated in $S_{g}^d$ for the dictionary of Legendre polynomials $\Psi_{\mathrm{Legendre}}$.
For our numerical experiments we chose $d=10$, $g=5$ and $\rho_{\mathrm{max}}=3$ and again compare the reconstruction w.r.t.\ $S_{g}^d$, the block-sparse TT representations of $S_{g,\rho_{\mathrm{max}}}^d$ and $S_{g,\rho_{\mathrm{max}}}^{d,\mathrm{aug}}$ and a dense TT representation of $T_r(V_{p}^d)$ with rank $r \le 14$.
Admittedly,  the choice $d=10$ is relatively small for this problem but was necessary since the computation on $S_{g}^d$ took prohibitively long for larger values.
A comparison of the degrees of freedom for the different ansatz spaces is given in Table~\ref{tbl:darcy} the relative errors to the exact solution are depicted in Figure~\ref{fig:darcy}.
In this plot we can recognize the general pattern that a lower number of parameters can be associated with an improved sample efficiency.
However, we also observe that for small $M$ the relative error for $S_{g,\rho}^{d}$ is smaller than for $S_{g,\rho}^{d,\mathrm{aug}}$.
We interpret this as a consequence of the regularity of $u$ since the alternating scheme for the optimization in $S_{g,\rho}^{d}$ favors lower degree polynomials by construction.
In spite of this success, we have to point out that optimizing over $S_{g,\rho}^{d}$ took about $10$ times longer than optimizing over $S_{g,\rho}^{d,\mathrm{aug}}$.
Finally, we observe that the recovery in $T_{14}(V_6^{10})$ produces unexpectedly large relative errors when compared to previous results in~\cite{eigel_non-intrusive_2019}.
This implies that the rank-adaptive algorithm from~\cite{eigel_non-intrusive_2019} must have a strong regularizing effect that improves the sample efficiency.

\begin{table}[!ht]
    \centering
    \bgroup
    \def\arraystretch{1.25}  %
    \begin{tabular}{ | l | l | l | l | l | }
        \hline
        \multicolumn{1}{ | c | }{$S_{5}^{10}$} & \multicolumn{1}{ c | }{$S_{5,3}^{10}$} & $S_{5,3}^{10,\mathrm{aug}}$  & \multicolumn{1}{ c | }{$T_{14}(V_{6}^{10})$}  & \multicolumn{1}{ c | }{$V_6^{10}$} \\
        \hline
        $3003$ & $1726$ & $803$ & $7896$ & $60466176$ \\
        \hline
    \end{tabular}
    \egroup
    \caption{Degrees of freedom for the full space $S_{g}^d$, the TT variant  $S_{g,\rho_{\mathrm{max}}}^d$ with maximal block size $\rho_{\mathrm{max}}=3$, the space $T_r(V_p^d)$ with TT rank bounded by $r=14$, and the full space $V_p^d$ for completeness.}
    \label{tbl:darcy}
\end{table}

\begin{figure}[!ht]
    \centering
    \includegraphics[width=0.75\textwidth]{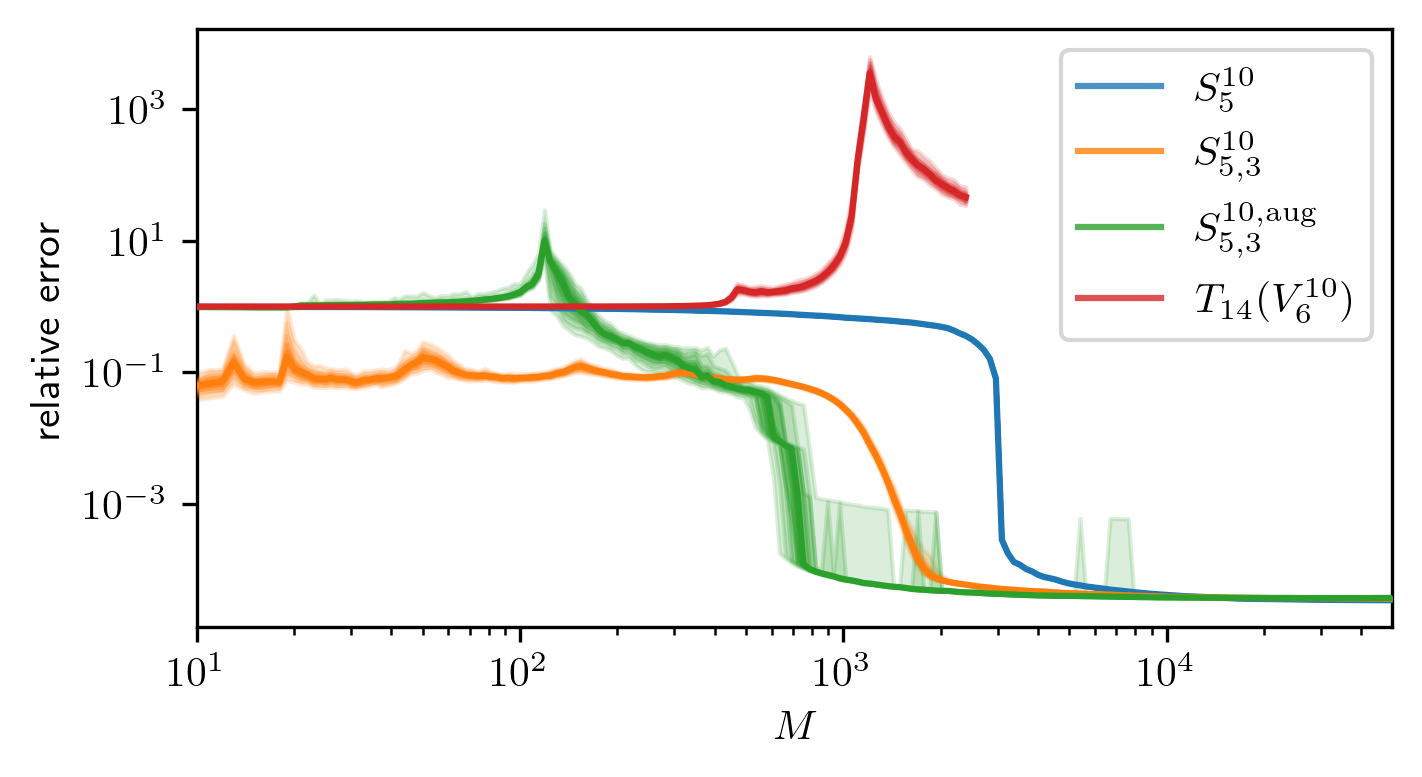}
    \caption{$0.15$--$0.85$ quantiles for recovery in blue: $S_{5}^{10}$, orange: $S_{5,3}^{10}$, green: $S_{5,3}^{10,\mathrm{aug}}$, and red: $T_{14}(V_6^{10})$.
    The experiment for $T_{14}(V_6^{10})$ was stopped early at $M=1200$ due to its prohibitive computational demand and because the expected behaviour is already observable.} \label{fig:darcy}
\end{figure}

\section{Conclusion}
We discuss the problem of function identification from data for tensor train based ansatz spaces and give some insights into when these ansatz spaces can be used efficiently.
For this we combine recent results on sample complexity~\cite{eigel_convergence_2020} and block sparsity of tensor train networks~\cite{bachmayr} to motivate a novel algorithm for the problem at hand.
We then demonstrate the applicability of this algorithm to different problems.
Up until know only dense tensor trains were used for these recovery tasks.
The numerical examples however demonstrate that this format can not compete with our novel block-sparse approach.
We observe that the sample complexity can be much more favorable for successful system identification with block sparse tensor trains than with dense tensor trains or purely sparse representations.
We expect that inclusion of rank-adaptivity using techniques from SALSA or bASD is straight forward and consider it an interesting direction for forthcoming papers.
We expect, that this would improve the numerical results even further.
The introduction of rank-adaptivity would moreover alleviate the problem of having to choose a block size a-priori. %
Finally, we want to reiterate that the spaces of homogeneous polynomials are predestined for the application of least-squares recovery with an optimal sampling density (cf.~\cite{cohen_optimal_2017}) which holds opportunities for further improvement of the sample efficiency.
This leads us to the conclusion that the proposed algorithm can be applied successfully to other high dimensional problems in which the sought function exhibits sufficient regularity. 

\section*{Acknowledgements}

M.\ G\"otte was funded by DFG (SCHN530/15-1).
R.\ Schneider was supported by the Einstein Foundation Berlin.
P. Trunschke acknowledges support by the Berlin International Graduate School in Model and Simulation based Research (BIMoS).

\bibliographystyle{alpha}
\bibliography{references.bib}

\end{document}